%% file: paper.tex
\documentclass[USenglish]{article}
\input{commands}

\usepackage{hyperref}

\hypersetup{
pdftitle={Some remarks about conservation and entropy stability for residual distribution schemes}
pdfauthor={R. Abgrall},
pdfpagemode=UseOutlines,
linkbordercolor=0 0 0,
linkcolor=red,
citecolor=blue,
colorlinks=true,
bookmarks = true
}
\begin{document}
\title{Some remarks about conservation %and entropy stability 
for residual distribution schemes}
\author{R. Abgrall.}
\date{\today}
\maketitle

\begin{abstract}
We are interested in the discretisation of the  steady version of hyperbolic problems. We first show that all the known schemes (up to our knowledge) can be rephrased in a common framework. Using this framework, we then show they  flux formulation, with an explicit construction of the flux, and thus are locally conservative. This is well known for the finite volume schemes or the discontinuous Galerkin ones, much less known for the continuous finite element methods. We also  show that Tadmor's entropy stability formulation can naturally be rephrased in this framework as an additional conservation relation discretisation, and using this, we show some connections with the recent papers \cite{ShuEntropy,Mishra,Gassner,Zingg}.
This contribution  is an enhanced version of \cite{icm}.
\end{abstract}
%\tableofcontents
\section{Introduction}
In this paper, we are interested in the approximation of non-linear hyperbolic problems.
 To make things more precise, our target are the  Euler equations in the compressible regime, other examples are the MHD equations. The case of parabolic problems in which the elliptic terms play an important role only in some area of the computational domain, such as the Navier-Stokes equations in the compressible regime, or the resistive MHD equations, can be dealt with in a similar way. 
In a series of papers \cite{Abgrall99,energie,AbgrallRoe,abgrall:dgrds,abgrall:shu,abgrallLarat,Mario,abgralldeSantisSISC,AbgralldeSantisNS,Abgrall2017}, following the pioneering work of Roe and Deconinck \cite{struijs}, we have  developed, with collaborators\footnote{in particular M. Ricchiuto, from INRIA Bordeaux Sud-Ouest},  a class of schemes that borrow some features from the finite element methods, and others, such as a  local maximum principle and a  non-linear stabilisation from the finite difference/finite volume methods. Though the methods have been developed with some rigour,  there is a lack of a more theoretical analysis, and also to explain in a clearer way the connections with more familiar methods such as the continuous finite elements methods or the discontinuous Galerkin ones.

The ambition of this paper is to provide this link through a discussion about conservation and entropy stability. In most of the paper, we consider  steady problems in the scalar case. The extension  to the system case is immediate. Examples of schemes are given in the paper and the appendix. %We indicate, in the end of this paper,  how to extend all this to unsteady problems. 
Their extensions to the system case can be found in \cite{abgrallLarat} for the pure hyperbolic case and in \cite{abgralldeSantisSISC,AbgralldeSantisNS} for the Navier Stokes equations. 

The model problem is 
\begin{subequations}\label{eq1}
\begin{equation}
\label{eq1:1}
\text{ div }\bbf(\bu)=0\qquad \text{in }\Omega
\end{equation}
subjected to
\begin{equation}\label{eq1:2}
\min( \nabla_\bu \bbf(\bu)\cdot \bn(\bx), 0) (\bu-\bu_b)=0 \text{ on }\partial\Omega.
\end{equation}
\end{subequations}
The domain $\Omega$ is assumed to be bounded, and regular. We assume for simplicity that its boundary is never characteristic. We also assume that it has a polygonal shape and thus any triangulation that we consider covers $\Omega$ exactly.
In \eqref{eq1:2}, $\bn(\bx)$ is the outward unit vector at $\bx\in \partial \Omega$ %(thus we assume enough regularity for $\Omega$) 
and $\bu_b$ is a  regular enough function.  The weak formulation of \eqref{eq1} is: $\bu\in  L^\infty(\Omega)$ is a weak solution of \eqref{eq1} if for any $\varphi\in C^1_0(\Omega)$, 
\begin{equation}
\label{weak:eq1}
-\int_\Omega \nabla \bv\cdot \bbf(\bu^h)\; d\bx+\int_{\partial \Omega} \bv \big ( \mathbf{\mathcal{F}}_\bn(\bu,\bu_b)-\bbf(\bu)\cdot\bn\big ) \; d\gamma=0
\end{equation}
where $\mathbf{\mathcal{F}}_\bn$ is a flux that is almost everywhere the upwind flux:
$$\mathbf{\mathcal{F}}_\bn(\bu,\bu_b)=\left \{\begin{array}{ll}
\bbf(\bu_b)\cdot\bn & \text{ if } \nabla_\bu \bbf(\bu)\cdot \bn >0\\
\bbf(\bu)\cdot\bn & \text{ else.}
\end{array}
\right .
$$

In a first part, we present the class of schemes (nicknamed as Residual Distribution Schemes or RD or RDS for short)  we are interested in, and show their link with more classical methods such as finite element ones. Then we recall  a condition that guarantees that the numerical solution will converge to a weak solution of the problem.
%, and a second one about entropy condition.
 In the third part, we show that the RD schemes are also finite volume schemes: we compute explicitly the flux. %There is not uniqueness. 
 In the fourth part, show that the now classical condition given by Tadmor in \cite{TadmorVieux,TadmorActa} in one dimension fits very naturally in our framework.
\section{Notations}
 From now on, we assume that $\Omega$ has a polyhedric boundary.  This simplification is by no mean essential. We denote by $\mathcal{E}_h$ the set of internal edges/faces of $\mathcal{T}_h$, and by $\mathcal{F}_h$ those contained in $\partial \Omega$.  $\KK$ stands either for an element $K$ or a face/edge $e\in \mathcal{E}_h\cup \mathcal{F}_h$. The boundary faces/edges are denoted by $\Gamma$.  The mesh is assumed to be shape regular, $h_K$ represents the diameter of the element $K$. Similarly, if $e\in \mathcal{E}_h\cup \mathcal{F}_h$, $h_e$ represents its diameter.

 Throughout this paper, we follow Ciarlet's definition \cite{ciarlet,ErnGuermond} of a finite element approximation: we have a set of degrees of freedom $\Sigma_K$ of linear forms acting on the set $\PP^k$ of polynomials of degree $k$ such that the linear mapping
 $$q\in \PP^k\mapsto \big (\sigma_1(q), \ldots, \sigma_{|\Sigma_K|}(q)\big )$$
 is one-to-one. The space $\PP^k$ is spanned by the basis function $\{\varphi_{\sigma}\}_{\sigma\in \Sigma_K}$  defined by
 $$\forall \sigma,\, \sigma',  \sigma(\varphi_{\sigma'})=\delta_\sigma^{\sigma'}.$$
  We have in mind either Lagrange interpolations where the degrees of freedom are associated to points in $K$, or other type of polynomials approximation such as B\'ezier polynomials where we will also do the same geometrical identification.
 Considering all the elements covering $\Omega$, the set of degrees of freedom is denoted by $\mathcal{S}$ and a generic degree of freedom  by $\sigma$. We note that for any $K$, 
 $$\forall \bx\in K, \quad \sum\limits_{\sigma\in K}\varphi_\sigma(\bx)=1.$$
 For any element $K$, $\#K$ is the number of degrees of freedom in $K$. If $\Gamma$ is a face or a boundary element, $\#\Gamma$ is also the number of degrees of freedom in $\Gamma$.

 The integer $k$ is assumed to be the same for any element.  We define 
$$\mathcal{V}^h=\bigoplus_K\{ \bv\in L^2(K), \bv_{|K}\in \PP^k\}.$$
The solution will be sought for in a  space $V^h$ that is:
\begin{itemize}
\item Either $V^h=\mathcal{V}^h$. In that case, the elements of $V^h$ can be discontinuous across internal faces/edges of $\mathcal{T}_h$. There is no conformity requirement on the mesh.
\item Or  $V^h=\mathcal{V}_h\cap C^0(\Omega)$ in which case the mesh needs to be conformal.
\end{itemize}

Throughout the text, we need to integrate functions. This is done via quadrature formula, and the symbol $\oint$ used in volume integrals
$$\oint_K v(\bx)\; d\bx$$
or boundary integrals
$$
\oint_{\partial K} v(\bx)\; d\gamma$$
means that these integrals are done via user defined numerical quadratures.

 If $e\in \mathcal{E}_h$, represents any  \emph{internal} edge, i.e. $e\subset K\cap K^+$ for two elements $K$ and $K^+$,  we define for any function $\psi$ the jump  $[\nabla \psi ]=\nabla \psi_{|K}-\nabla \psi_{| K^+}$. Here the choice of $K$ and $K^+$ is important, hence  also see relation \eqref{jump} in section \ref{sec:multiD} where these element are defined in the relevant context. Similarly, $\{\bv\}=\tfrac{1}{2}\big (\bv_{|K}+\bv_{|K^+}\big )$.
 
 If $\mathbf{x}$ and $\mathbf{y}$ are two vectors of $\R^q$, for $q$ integer, $\langle \bx,\by\rangle$ is their scalar product. In some occasions, it can also be denoted as $\bx\cdot\by$ or $\bx^T\by$.  We also use  $\bx\cdot\by$ when $\bx$ is a matrix and $\by$ a vector: it is simply the matrix-vector multiplication.

In sections \ref{flux} and \ref{sec:entropy}, we have to deal with oriented graph. Given two vertices of this graph $\sigma$ and $\sigma'$, we write $\sigma>\sigma'$ to say that $[\sigma,\sigma']$ is a direct edge.

     \section{Schemes and conservation}
     \input{schemes}

\input{conservation}

\input{flux}

\input{examples}

\input{entropy_new}

%\input{Numeric}
%\input{unsteady}
\section{Conclusion}
This paper shows some links between now classical schemes, such as the finite volume scheme, the continuous finite element methods, the discontinuous Galerkin methods and more generally a class of method nicknamed as Residual Distribution (RD) methods. We show that, under a proper definition of a consistent flux, all these schemes enjoy a flux formulation, and hence are locally conservative. This is well known for most schemes, less  known for some of them. The fluxes are explicitly given. 
We also show that Tadmor's entropy stability condition can  be reformulated very simply in the Residual Distribution context. Using this we have shown some connections with the recent work \cite{ShuEntropy}. However the discussion here is certainly not finished, it will be the topic of another paper.

The emphasis of this paper is put on the steady case, but the unsteady state is similar, see \cite{Mario} and \cite{Abgrall2017}. 
\section*{Acknowledgements}
The author has been funded in part  by the SNSF project 200021\_153604 "High fidelity simulation for compressible materials". I would also like to thanks Anne Burbeau 
(CEA-DEN) for her critical reading of the first draft of this paper. Her input has hopefully helped to improve the readability of this paper. The two referees and the editor are also warmly thanked for their patience, their comments and ability to trace typos. The remaining mistakes are mine. 
\bibliographystyle{unsrt}
\bibliography{papier}
\appendix
\input{DG_RDS}
\end{document}

%% file: commands.tex
\usepackage{fullpage}

\usepackage{amsmath,amsfonts,amsthm,amssymb}

\newtheorem{theorem}{Theorem}[section]

\newtheorem{lemma}[theorem]{Lemma}
\newtheorem{proposition}[theorem]{Proposition}
\newtheorem{remark}[theorem]{Remark}
\newtheorem{definition}[theorem]{Definition}

\usepackage{color,xcolor}
\usepackage{ifpdf}

\usepackage{psfrag}
\usepackage{graphicx,graphics,subfigure}
\usepackage{url}

 \newcommand{\R}{\mathbb R}

\newcommand{\PP}{\mathbb P}

\newcommand{\normal}{\mathbf{N}}
 \newcommand{\KK} {\mathcal{K}}

\newcommand{\bbf}{{\mathbf {f}}}

\newcommand{\bn}{{\mathbf {n}}}

\newcommand{\Bf}{\mathbf{f}}

\newcommand{\bu}{u}%\mathbf{u}}
\newcommand{\bv}{v}%\mathbf{v}}
\newcommand{\bF}{\mathbf{f}}
\newcommand{\bbF}{\mathbf{\mathcal{F}}}

\newcommand{\bbu}{u}%\mathbf{u}}
\newcommand{\bbg}{\mathbf{g}}
\newcommand{\hbbg}{\hat{\mathbf{g}}}
\newcommand{\hbbf}{\hat{\mathbf{f}}}

\newcommand{\bx}{\mathbf{x}}
\newcommand{\by}{\mathbf{y}}

%% file: schemes.tex
%        %%%%%%%%%%%%%%%%%
\subsection{Schemes}
We begin this section by recalling the notion of flux.
Let us  consider any common edge or face $\Gamma$ of $K^+$ and $K^-$, two elements. Let $\bn$ be the normal to $\Gamma$, see Figure \ref{fig:flux}. Depending on the context, $\bn$ is a scaled normal or $||\bn||=1$.
\begin{figure}[h]
\begin{center}
\psfrag{K+}{$K^+$}
\psfrag{K-}{$K^-$}
\psfrag{n}{$\bn$}
\includegraphics[width=0.45\textwidth]{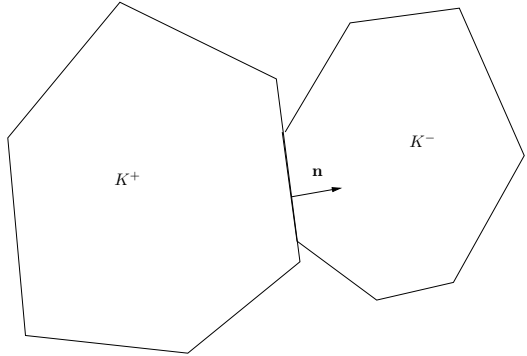}
\caption{\label{fig:flux}Geometrical setting}
\end{center}
\end{figure}
The symbols $S^\pm$ represent set of states, where $S^+$ is associated to $K^+$ and $S^-$ to $K^-$.
A flux $\hbbf_\bn(S^+, S^-)$ between $K^+$ and $K^-$ has to satisfy
\begin{subequations}\label{flux:def}
\begin{equation}\label{flux:def:1}
\hbbf_\bn(S^+, S^-)=-\hbbf_{-\bn}(S^-, S^+).
\end{equation}
and the consistency condition when the sets $S^\pm$ reduce to $\bu$
\begin{equation}\label{flux:def:2}
\hbbf_\bn(S,S)=\bbf(\bu)\cdot \bn.
\end{equation}
\end{subequations}
 For a first order finite volume scheme,
we have $S^+=\bu_{K_+}$ and $S^-=\bu_{K_-}$,  the average values of $\bu$ in $K^+$ and $K^-$. For the other schemes, for example high order schemes,  the definition is more involved.

{ In order to integrate the steady version of \eqref{eq1} on a domain $\Omega\subset \R^d$ with the boundary conditions \eqref{eq1:2}, on each element $K$ and any degree of freedom $\sigma\in \mathcal{S}$ belonging to $K$,  we define residuals $\Phi_\sigma^K(\bu^h)$. Following \cite{abgrallLarat,abgralldeSantisSISC}, they are assumed to satisfy the following conservation relations:
For any element $K$, 
\begin{equation}
\label{conservation:K}
\sum\limits_{\sigma\in K}\Phi_\sigma^K(\bu^h)=\int_{\partial K}\hbbf_\bn(\bu^h,\bu^{h,-})\; d\gamma ,
\end{equation}
where $\bu^{h,-}$ is the approximation of the solution on the other side of the local edge/face of $K$.
 Note that in the case of a conformal mesh and  with globally continuous elements, the condition reduces to
$$\sum\limits_{\sigma\in K}\Phi_\sigma^K(\bu^h)=\int_{\partial K} \bbf(\bu^h)\cdot\bn \; d\gamma .$$
Similarly, we consider residuals on the boundary elements $\Gamma$. On any such $\Gamma$, for any degree of freedom $\sigma\in \mathcal{S}\cap \Gamma$, we consider boundary residuals $\Phi_\sigma^\Gamma(u^h)$ that will satisfy the conservation relation
\begin{equation}
\label{conservation:Gamma}
\sum\limits_{\sigma\in \Gamma}\Phi_\sigma^\Gamma(\bu_h)=\int_{\Gamma}\big ( \mathcal{F}_\bn(\bu^h, \bu_b)-\bF(\bu^h)\cdot \bn\big ) \; d\gamma.
\end{equation}
Once this is done, the discretisation of \eqref{eq1} is achieved via: for any $\sigma\in \mathcal {S}$,
\begin{equation}
\label{RD:scheme}
\sum\limits_{K\subset\Omega, \sigma\in K}\Phi_\sigma^K(\bu^h)+\sum\limits_{\Gamma\subset\partial\Omega, \sigma\in \Gamma}\Phi_\sigma^\Gamma(\bu^h)=0.
\end{equation}
In \eqref{RD:scheme}, the first term represents the contribution of the internal elements. The second exists if $\sigma\in \partial \Omega$ and represents the contribution of the boundary conditions.

In fact, the formulation \eqref{RD:scheme} is very natural. Consider a variational formulation of the steady version of \eqref{eq1}:
$$\text{ find } \bu^h\in V^h \text{ such that for any } {v}^h\in V^h,  a(\bu^h,\bv^h)=0.$$
Let us show on  three examples that this variational formulation leads to \eqref{RD:scheme}. They are
\begin{itemize}
\item The SUPG \cite{Hughes1} variational formulation,  with $\bu^h, \bv^h\in V^h=\mathcal{V}^h\cap C^0(\Omega)$:
\begin{equation}
\label{SUPG:var}
\begin{split}
a(\bu^h,\bv^h)&:=-\int_\Omega \nabla \bv^h\cdot \bF(\bu^h)\; d\bx+\sum\limits_{K\subset \Omega}h_K\int_K\big [ \nabla \bF(\bu^h)\cdot\nabla \bv^h\big ] \; \tau_K\; \big [\nabla\bF(\bu^h)\cdot \nabla \bu^h\big ] d\bx\\
&\qquad +\int_{\partial \Omega} \bv^h \big (\bbF_\bn(\bu^h,\bu_b)-\bF(\bu^h)\cdot\bn\big ) \;d\gamma .
\end{split}
\end{equation}
Here $\tau_K$ is a positive parameter.
\item The Galerkin scheme with jump stabilization,  see \cite{burman} for details.  We have
\begin{equation}
\label{burman:var}
\begin{split}
a(\bu^h,\bv^h)&:=-\int_\Omega \nabla \bv^h\cdot \bF(\bu^h)\; d\bx+\sum\limits_{e \subset \Omega}\theta_e h_e^2\int_e \big [ \nabla \bv^h \big ]\cdot \big [ \nabla \bu^h\big ] \; d\gamma \\
&\qquad +\int_{\partial \Omega} \bv^h \big (\bbF_\bn(\bu^h,\bu_b)-\bF(\bu^h)\cdot\bn\big ) \; d\gamma .
\end{split}
\end{equation}
 Here,  $\bu^h, \bv^h\in V^h=\mathcal{V}^h\cap C^0(\Omega)$, and $\theta_e$ is a positive parameter.
\item The discontinuous Galerkin formulation: we look for $\bu^h, \bv^h\in V^h=\mathcal{V}^h$ such that
\begin{equation}\label{DG:var}
a(\bu^h,\bv^h):=\sum\limits_{K\subset \Omega}\bigg ( -\int_K\nabla\bv^h\cdot\bbf(\bu^h) d\bx+\int_{\partial K}\bv^h\cdot \hat{\bbf}_\bn(\bu^{h},\bu^{h,-}) \;d\gamma \bigg ).
\end{equation}
In \eqref{DG:var}, the boundary integral is a sum of integrals on the faces of $K$, and here for any face of $K$
 $\bu^{h,-}$ represents the approximation of $\bu$ on the other side of that face in the case of internal elements, and $\bu_b$ when that face is on $\partial \Omega$. % In \eqref{DG:var}, $\hbbf$ is a consistent numerical flux. 
 Note that to fully comply with \eqref{RD:scheme}, we should have defined for boundary faces $\bu^{h,-}=\bu^h$, and then \eqref{DG:var} is rewritten as 
 \begin{equation}\label{DG:var:2}
a(\bu^h,\bv^h):=\sum\limits_{K\subset \Omega}\bigg ( -\int_K\nabla\bv^h\cdot\bbf(\bu^h) d\bx+\int_{\partial K}\bv^h\, \hat{\bbf}_\bn(\bu^{h},\bu^{h,-}) \; d\gamma \bigg )+\sum\limits_{\Gamma\subset\partial\Omega}\int_{\Gamma}\bv^h\cdot\bigg ( \bbF_\bn(\bu^h,\bu_b)-\bbf(\bu^h)\cdot \bn \bigg )\; d\gamma.
\end{equation}
In \eqref{DG:var}, we have implicitly assumed $\hbbf_\bn=\bbF_\bn$ on the boundary edges.
\end{itemize}
In the SUPG, Galerkin scheme with jump stabilisation or the DG scheme, the boundary flux can be chosen different from $\bbF$. This can lead to boundary layers if these flux are not "enough" upwind, but we are not interested in these issues here.

Using the fact that the basis functions that span $V_h$ have a \emph{compact} support, then each scheme can be rewritten in the form \eqref{RD:scheme} with the following expression for the residuals:
    \begin{itemize}
    \item For the  SUPG scheme \eqref{SUPG:var}, the  residual are defined by
    \begin{equation}\label{SUPG}\Phi_\sigma^K(\bu^h)=\int_{\partial K}\varphi_\sigma \bF(\bu^h)\cdot \bn \; d\gamma -\int_K \nabla \varphi_\sigma\cdot \bF(\bu^h) \; d\bx+h_K
    \int_K \bigg (\nabla_\bu\bF(\bu^h)\cdot \nabla \varphi_\sigma \bigg )\tau_K \bigg (\nabla_\bu\bF(\bu^h)\cdot \nabla \bu^h \bigg )\;d\bx .\end{equation}
  %  with $\tau_K>0$.
    \item For the Galerkin scheme with jump stabilization \eqref{burman:var}, the residuals are defined  by:  
        \begin{equation}\label{burman}\Phi_\sigma^K(\bu^h)=\int_{\partial K}\varphi_\sigma \bF(\bu^h)\cdot \bn\; d\gamma -\int_K \nabla \varphi_\sigma\cdot \bF(\bu^h)\; d\bx +
    \sum\limits_{e \text{ faces of }K} \frac{\theta_e}{2} h_e^2 \int_{\partial K} [\nabla \bu^h]\cdot [\nabla \varphi_\sigma]\; d\gamma\end{equation}
    with $\theta_e>0$.
    Here, since the mesh is conformal, any internal edge $e$ (or face in 3D) is the intersection of the element $K$ and another element denoted by $K^+$.
    \item For the discontinuous Galerkin scheme,
    \begin{equation}\label{DG}
    \Phi_\sigma^K(\bu^h)=-\int_K\nabla\varphi_\sigma\cdot\bbf(\bu^h) d\bx+\int_{\partial K}\varphi_\sigma\cdot \hat{\bbf}_\bn(\bu^{h},\bu^{h,-}) \; d\gamma
    \end{equation}
    using the second definition of $\bu^{h,-}$.
    \item The boundary residuals are 
  \begin{equation}
  \label{boundary}
  \Phi_\sigma^\Gamma(\bu^h)=  \int_{\Gamma}\varphi_\sigma\big ( \mathcal{F}_\bn(\bu^h, \bu_b)-\bF(\bu^h)\cdot \bn\big )\; d\gamma
  \end{equation}
  \end{itemize}
All these residuals satisfy the relevant conservation relations, namely \eqref{conservation:K} or \eqref{conservation:Gamma}, depending if we are dealing with element residuals or boundary residuals.

\bigskip

For now, we are just rephrasing classical finite element schemes into a purely numerical framework. However, considering  the pure numerical point of view and forgetting the variational framework, we can go further and define schemes that have no clear variational formulation. 
These are  the limited  Residual Distributive Schemes, see \cite{abgrallLarat,abgralldeSantisSISC}, namely
    \begin{equation}
    \label{schema RDS}\Phi_\sigma^K(\bu^h)=\beta_\sigma \int_{\partial K}\bF(\bu^h)\cdot \bn\; d\gamma
    \end{equation}
    or 
    \begin{equation}
    \label{schema RDS SUPG}\Phi_\sigma^K(\bu^h)=\beta_\sigma \int_{\partial K}\bF(\bu^h)\cdot \bn\; d\gamma+\theta_Kh_K
    \int_K \bigg (\nabla_\bu\bF(\bu^h)\cdot \nabla \varphi_\sigma \bigg )\tau_K \bigg (\nabla_\bu\bF(\bu^h)\cdot \nabla \bu^h \bigg )\;d\bx, \qquad \theta_K\geq 0
    \end{equation}
or
\begin{equation}
\label{schema RDS jump}\Phi_\sigma^K(\bu^h)=\beta_\sigma \int_{\partial K}\bF(\bu^h)\cdot \bn\; d\gamma+
    \theta_e \;h_e^2 \int_{\partial K} [\nabla \bu^h]\cdot [\nabla \varphi_\sigma]\; d\gamma \qquad \theta_e\geq 0\end{equation}
   where the parameters $\beta_\sigma$ are defined to guarantee conservation,
   $$\sum\limits_{\sigma\in K} \beta_\sigma=1$$
   and such that \eqref{schema RDS SUPG} without the streamline term and \eqref{schema RDS jump} without the jump term satisfy a discrete maximum principle. The streamline term and jump term are introduced because one can easily see that spurious modes may exist, but their role is very different compared to \eqref{SUPG} and \eqref{burman} where they are introduced to stabilize the Galerkin scheme: if formally the maximum principle is violated, experimentally the violation is extremely small if existent at all. See \cite{energie,abgrallLarat} for more details. 
   
   A similar construction can be done starting from a discontinuous Galerkin scheme, see \cite{abgrall:shu,abgrall:dgrds}. 
   A second order version is described in appendix \ref{DG_RDS}.
   
   The non-linear stability is provided by the coefficient $\beta_\sigma$ which is a non-linear function of $\bu^h$.  Possible values of $\beta_\sigma$ are described in remark \ref{beta} bellow.
   \begin{remark}\label{beta}
   The coefficients $\beta_\sigma$ introduced in the relations \eqref{schema RDS SUPG} and \eqref{schema RDS jump} are defined by:
\begin{equation}\label{eqbeta}
\beta_\sigma=\dfrac{\max(0,\frac{\Phi_\sigma}{\Phi})}
{
\sum\limits_{\sigma'\in K} \max(0,\frac{\Phi_{\sigma'}}{\Phi})}.
\end{equation}
 These coefficients are always defined and   garantee a local maximum principle for \eqref{schema RDS SUPG} and \eqref{schema RDS jump}: this is again a consequence of the conservation properties,  see e.g. \cite{abgrallLarat}. Note this is true for any order of interpolation.
 \end{remark}
   
  }

%% file: conservation.tex
\subsection{Conservation}\label{sec:conservation}
From \eqref{RD:scheme}, using the conservation relations \eqref{conservation:Gamma} and \eqref{conservation:K}, we obtain for any $\bv^h\in V^h$,
$$\bv_h=\sum_{\sigma \in \mathcal{S}} \bv_\sigma \varphi_\sigma,$$ the following relation:
\begin{equation}\label{algebre2}
\begin{split}
0&=  -\int_\Omega \nabla \bv_h \cdot \bbf(\bu^h) \; d\bx  +\int_{\partial \Omega}\bv^h \big (\hat{\bbf}_\bn(\bu^h,\bu_b)-\bbf(\bu^h)\cdot \bn\big ) \; d\gamma\\
&\qquad +\sum\limits_{e\in \mathcal{E}_h} \int_e[\bv^h]\hbbf_\bn(\bu^h,\bu^{h,-})\; d\gamma+\sum\limits_{K\subset \Omega}\frac{1}{\#K}\bigg ( \sum\limits_{\sigma,\sigma'\in K} (\bv_\sigma-\bv_{\sigma'})\bigg ( \Phi_\sigma^K(\bu^h)-\Phi_\sigma^{K, Gal}(\bu^h) \bigg ) \bigg )\\
&\qquad \qquad
+\sum\limits_{\Gamma\subset \partial \Omega} \frac{1}{\#\Gamma}\bigg ( \sum\limits_{\sigma,\sigma'\in \Gamma} (\bv_\sigma-\bv_{\sigma'})(\Phi_\sigma^\Gamma \big (\bu^h,\bu_b)-\Phi_\sigma^{Gal,\Gamma}(\bu^h,\bu_b)\big )\bigg )
\end{split}
%-\int_\Omega \nabla \bv^h\cdot \bbf(\bu^h)d\bx+\int_{\partial \Omega} \bv^h\cdot \hat{\bbf}_\bn(\bu^h, \bu_b) d\gamma &+
%\sum\limits_{K}\sum_{\sigma, \sigma'\in K} \big (\bv_\sigma-\bv_{\sigma'}) \cdot \big(\Phi_{\sigma}^K(\bu^h)-\Phi_\sigma^{K, Gal}\big )\\
%&\qquad+
%\sum_{\Gamma\subset\partial \Omega} \sum_{\sigma, \sigma'\in \Gamma} \big (\bv_\sigma-\bv_{\sigma'}) \cdot \big(\Phi_{\sigma}^\Gamma(\bu^h)-\Phi_\sigma^{\Gamma, Gal}\big )=0
%\end{split}
\end{equation}
where
$$\Phi_\sigma^{K, Gal}(\bu^h)=-\int_K\nabla\varphi_\sigma\cdot \bbf(\bu^h) \; d\bx +\int_{\partial K} \varphi_\sigma \hbbf_\bn(\bu^h,\bu^{h,-}) \; d\gamma,
\qquad \Phi_\sigma^{\Gamma, Gal}(\bu^h,\bu_b)=\int_\Gamma \varphi_\sigma \big (\hat{\bbf}_\bn(\bu^h,\bu_b)-\bbf(\bu^h)\cdot \bn\big ) \; d\gamma.$$
\begin{proof}
We start from \eqref{RD:scheme} which is multiplied by $\bv_\sigma$, and these relations are added for each $\sigma\in \mathcal{S}$. We get:
$$
0=\sum\limits_{\sigma\in \mathcal{S}}\bv_\sigma\bigg ( \sum\limits_{K\subset \Omega, \sigma\in K} \Phi_\sigma^K(\bu^h)+\sum\limits_{\Gamma\subset\partial \Omega, \sigma\in \Gamma}\Phi_\sigma^\Gamma(\bu^h,\bu_b) \bigg ).
$$
Permuting the sums on $\sigma$ and $K$, then on $\sigma$ and $\Gamma$, we get:
$$
0=\sum\limits_{K\subset \Omega} \bigg ( \sum\limits_{\sigma\in K} \bv_\sigma\Phi_\sigma^K(\bu^h)\bigg ) + \sum\limits_{\Gamma\subset\partial \Omega}\bigg ( \sum\limits_{\sigma\in \Gamma} \bv_\sigma\Phi_\sigma^\Gamma(\bu^h,\bu_b)\bigg ) .$$
We look at the first term, the second is done similarly.
We have, introducing $\Phi_\sigma^{K, Gal}$ and $\#K$ the number of degrees of freedom in $K$,
\begin{equation*}
\begin{split}
\sum\limits_{\sigma\in K} \bv_\sigma\Phi_\sigma^K(\bu^h)&=\sum\limits_{\sigma\in K} \bv_\sigma \Phi_\sigma^{K, Gal}(\bu^h) +\sum\limits_{\sigma\in K} \bv_\sigma\bigg ( \Phi_\sigma^K(\bu^h)-\Phi_\sigma^{K, Gal}(\bu^h) \bigg )\\
&= -\int_K \nabla \bv_h \cdot \bbf(\bu^h) \; d\bx +\int_{\partial K} \bv^h \hbbf_\bn(\bu^h, \bu^{h,-})\; d\gamma+ \sum\limits_{\sigma\in K} \bv_\sigma\bigg ( \Phi_\sigma^K(\bu^h)-\Phi_\sigma^{K, Gal}(\bu^h) \bigg )\\
&=-\int_K \nabla \bv_h \cdot \bbf(\bu^h) \; d\bx +\int_{\partial K} \bv^h \hbbf_\bn(\bu^h, \bu^{h,-})\; d\gamma+ \frac{1}{\#K}\sum\limits_{\sigma,\sigma'\in K} (\bv_\sigma-\bv_{\sigma'})\bigg ( \Phi_\sigma^K(\bu^h)-\Phi_\sigma^{K, Gal}(\bu^h) \bigg )
\end{split}
\end{equation*}
because $$\sum\limits_{\sigma\in K} \big ( \Phi_\sigma^K(\bu^h)-\Phi_\sigma^{K, Gal}(\bu^h) \big )=0.$$

Similarly, we have
\begin{equation*}
\begin{split}
\sum\limits_{\sigma\in \Gamma } \bv_\sigma\Phi_\sigma^\Gamma(\bu^h)&=\int_{\Gamma}\bv^h \big (\hat{\bbf}_\bn(\bu^h,\bu_b)-\bbf(\bu^h)\cdot \bn\big ) \; d\gamma+\sum\frac{1}{\#\Gamma}\sum\limits_{\sigma,\sigma'\in \Gamma} (\bv_\sigma-\bv_{\sigma'})(\Phi_\sigma^\Gamma \big (\bu^h,\bu_b)-\Phi_\sigma^{Gal,\Gamma}(\bu^h,\bu_b)\big )
\end{split}
\end{equation*}
Adding all the relations, we get:
\begin{equation*}
\begin{split}
0&= \sum\limits_{K\subset \Omega} \bigg ( -\int_K \nabla \bv_h \cdot \bbf(\bu^h) \; d\bx +\int_{\partial K} \bv^h \hbbf_\bn(\bu^h, \bu^{h,-})\; d\gamma\bigg ) +\sum\limits_{\Gamma\subset \partial \Omega}\int_{\Gamma}\bv^h \big (\hat{\bbf}_\bn(\bu^h,\bu_b)-\bbf(\bu^h)\cdot \bn\big ) \; d\gamma\\
& \qquad  +\sum\limits_{K\subset \Omega}\frac{1}{\#K}\bigg ( \sum\limits_{\sigma,\sigma'\in K} (\bv_\sigma-\bv_{\sigma'})\bigg ( \Phi_\sigma^K(\bu^h)-\Phi_\sigma^{K, Gal}(\bu^h) \bigg ) \bigg )\\
&\qquad \qquad
+\sum\limits_{\Gamma\subset \partial \Omega} \frac{1}{\#\Gamma}\bigg ( \sum\limits_{\sigma,\sigma'\in \Gamma} (\bv_\sigma-\bv_{\sigma'})(\Phi_\sigma^\Gamma \big (\bu^h,\bu_b)-\Phi_\sigma^{Gal,\Gamma}(\bu^h,\bu_b)\big )\bigg )
\end{split}
\end{equation*}
i.e. after having defined $[\bv^h]= \bv^h-\bv^{h,-}$ and chosen one orientation of the internal edges $e\in \mathcal{E}_h$,  we get \eqref{algebre2}.
\end{proof}

The relation \eqref{algebre2} is instrumental in proving the following results.
The first one is proved in  \cite{AbgrallRoe}, and is a generalisation of the classical Lax-Wendroff theorem.
\begin{theorem}\label{th:LW}
Assume the family of meshes $\mathcal{T}=(\mathcal{T}_h)$
%_{h\in \HH}$ 
is shape regular. We assume that the residuals $\{\Phi_\sigma^{\mathcal{K}}\}_{\sigma\in \KK}$,
 for $\mathcal{K}$ an element or a boundary element of $\mathcal{T}_h$, satisfy: \begin{itemize}
\item For any $M\in \R^+$, there exists a constant $C$ which depends only on the family of meshes $\mathcal{T}_h$ and $M$ such that
 for any $\bu^h\in V^h$ with $||\bu^h||_{\infty}\leq M$, then
$$\big|\Phi^\KK_\sigma({\bu^h}_{|\KK})\big |\leq C\sum_{\sigma, \sigma'\in \KK}|\bu_\sigma^h-\bu_{\sigma'}^h|$$
\item The conservation relations \eqref{conservation:K} and \eqref{conservation:Gamma}.
\end{itemize}
Then if there exists a constant $C_{max}$ such that the solutions of the scheme \eqref{RD:scheme} satisfy $||\bu^h||_{\infty}\leq C_{max}$ and a function $\bv\in L^2(\Omega)$ such that $(\bu^h)_{h}$ or at least a sub-sequence converges to $\bv$ in $L^2(\Omega)$, then $\bv$ is a weak solution of \eqref{eq1}
\end{theorem}
\begin{proof}
The proof can be found in \cite{AbgrallRoe}, it uses \eqref{algebre2} and some adaptation of the ideas of \cite{kroner}. One of the key arguments comes from the consistency of the flux $\hbbf$ as well as \eqref{flux:def:1}\end{proof}

Another consequence of \eqref{algebre2} is the following result on entropy inequalities:
\begin{proposition}\label{th:entropy}
Let $(U,\mathbf{g})$ be a  entropy-flux couple for \eqref{eq1} and $\hat{\mathbf{g}}_\bn$ be a numerical entropy flux consistent with $\mathbf{g}\cdot \bn$. Assume that the residuals satisfy:
for any element $K$,
\begin{subequations}\label{entropy}
\begin{equation}\label{entropy:1}
\sum_{\sigma \in K}\langle\nabla_\bu U(\bu_\sigma), \Phi_\sigma^K\rangle \geq \int_{\partial K} \hat{\mathbf{g}}_\bn(\bu^h,\bu^{h,-}) \; d\gamma
\end{equation}
and for any boundary edge $e$,
\begin{equation}\label{entropy:2}
\sum_{\sigma \in e}\langle\nabla_\bu U(\bu_\sigma) ,  \Phi_\sigma^e\rangle  \geq \int_{e} \big (\hat{\mathbf{g}}_\bn(\bu^h,\bu_b)- \mathbf{g}(\bu^h)\cdot \bn \big )\; d\gamma.
\end{equation}
\end{subequations}
Then, under the assumptions of theorem \ref{th:LW}, the limit weak solution also satisfies the following entropy inequality: for any $\varphi\in C^1(\overline{\Omega})$, $\varphi\geq 0$, 
$$-\int_\Omega \nabla \varphi\cdot \mathbf{g}(\bu) \; d\bx+\int_{\partial\Omega^-}\varphi\;\mathbf{g}(u_b)\cdot \bn  \; d\gamma \leq 0.$$
\end{proposition}
\begin{proof}
The proof is similar to that of theorem \ref{th:LW}.
\end{proof}

\bigskip
Another consequence of \eqref{algebre2} is the following condition under which  one can guarantee to have a $k+1$-th order accurate scheme. We first introduce the (weak) truncation error
\begin{equation}\label{truncation}
\mathcal{E}\bigl(\bu^h, \varphi\bigr) = 
\sum_{\sigma \in \mathcal{S}_h}\varphi_\sigma\bigg [
\sum_{K\subset \Omega, \sigma\in K} \Phi_\sigma^K+\sum_{\Gamma\subset \partial\Omega, \sigma\in \Gamma}\Phi_\sigma^\Gamma\bigg ].
\end{equation}

If the solution of the \emph{steady problem} $\bu$ is smooth enough and the residuals,
computed with the interpolant $\pi_h(\bu)$ of the solution, are such that for any element $K$ and boundary element $\Gamma$
\begin{equation}\label{eq: residual accuracy}
%\begin{split}
\Phi_\sigma^K(\pi_h(\bu)) = \mathcal{O}\bigl(h^{k+d}\bigr), \qquad 
\Phi_\sigma^\Gamma(\pi_h(\bu))=\mathcal{O}\bigl(h^{k+d-1}\bigr)
\end{equation}
and if the approximation $\bbf\bigl(\bu^h\bigr)$ of $\bbf(\bu)$ is accurate with  order 
$k+1$, then the truncation error satisfies the following relation
\begin{equation*}
|\mathcal{E}\bigl(\pi_h(\bu), \varphi\bigr)| \le C( \bbf, \bu)\; ||\varphi||_{H^1(\Omega)}\; h^{k+1},
\end{equation*}
with $C$ a constant which depends only on $\bbf$, and $||\bu||_\infty$.
\begin{proof}
We first show that $\Phi_\sigma^{K, Gal}(\pi_h(\bu))=\mathcal{O}\bigl(h^{k+d}\bigr)$. Since $\bu$ is regular enough, we have pointwise $\text{ div } \bbf(\bu) =0$ on $K$, so that, by consistency of the flux, 
$$
0=-\int_K \nabla \varphi\cdot \bbf(\bu)\; d\bx+\int_{\partial K} \varphi \hbbf_\bn(\bu, \bu)\; d\gamma.
$$
Then,
\begin{equation*}
\begin{split}
\Phi_\sigma^{K, Gal}(\pi_h(\bu))&=-\int_K\nabla\varphi_\sigma\cdot \big (\bbf(\pi_h(\bu))-\bbf(\bu)\big ) \; d\bx +\int_{\partial K}\varphi_\sigma\big ( \hbbf_\bn(\pi_h(\bu), \pi_h(\bu))-\hbbf_\bn(\bu, \bu)\big )\; d\gamma\\
& = |K| \times \mathcal{O}(h^{-1}) \times \mathcal{O}(h^{k+1}) + |\partial K|\times  \mathcal{O}(1 )\mathcal{O}(h^{k+1})\\
&=\mathcal{O}(h^{d})\times \mathcal{O}(h^{-1}) \times \mathcal{O}(h^{k+1})+\mathcal{O}(h^{d-1})\times  \mathcal{O}(1 )\times\mathcal{O}(h^{k+1})\\
&=\mathcal{O}(h^{d+k})
\end{split}
\end{equation*}
because the flux is Lipschitz continuous and the mesh is regular.

The result on the boundary term is similar since  the boundary numerical flux is upwind and the boundary of $\Omega$ is not characteristic: only two types of boundary faces exists, the upwind and downwind ones. On the downwind faces, the boundary flux vanishes. On the upwind ones, we get the estimate  for the Galerkin boundary residuals thanks to the  same approximation argument.

The mesh is assumed to be regular: the number of elements (resp. edges) is $O(h^{-d})$ (resp. $O(h^{-d+1})$). 
Let us assume \eqref{eq: residual accuracy}. Let $\bv\in C_0^1(\overline{\Omega})$. Using \eqref{algebre2} for $\pi_h(\bu)$,\begin{equation*}
\begin{split}
\mathcal{E}\bigl(\bu^h, \varphi\bigr) =&  -\int_\Omega \nabla \pi_h(v) \cdot \bbf(\pi_h(\bu)) \; d\bx  +\int_{\partial \Omega}\pi_h(v) \big (\hat{\bbf}_\bn(\pi_h(\bu),\bu_b)-\bbf(\pi_h(\bu))\cdot \bn\big ) \; d\gamma\\
&\qquad +\sum\limits_{e\in \mathcal{E}_h} \int_e[\pi_h(v)]\hbbf_\bn(\pi_h(\bu),\pi_h(\bu)^{-})\; d\gamma+\sum\limits_{K\subset \Omega}\frac{1}{\#K}\bigg ( \sum\limits_{\sigma,\sigma'\in K} (\bv_\sigma-\bv_{\sigma'})\bigg ( \Phi_\sigma^K(\pi_h(\bu))-\Phi_\sigma^{K, Gal}(\pi_h(\bu)) \bigg ) \bigg )\\
&\qquad \qquad
+\sum\limits_{\Gamma\subset \partial \Omega} \frac{1}{\#\Gamma}\bigg ( \sum\limits_{\sigma,\sigma'\in \Gamma} (\bv_\sigma-\bv_{\sigma'})(\Phi_\sigma^\Gamma \big (\pi_h(\bu),\bu_b)-\Phi_\sigma^{Gal,\Gamma}(\pi_h(\bu),\bu_b)\big )\bigg ).
\end{split}
\end{equation*}
 where $\pi_h(\bu)^{-}$ represents the interpolant of $\bu$ on $K^-$.

We have, using 
$$-\int_\Omega \nabla \pi_h(v)\cdot \bbf(\bu)\; d\bx+\int_{\partial \Omega} \pi_h(v) \big (\mathcal{F}_\bn(\bu,\bu_b)-\bbf(\bu)\cdot \bn\big ) \; d\gamma=0,$$
\begin{equation*}
\begin{split}
-\int_\Omega \nabla \pi_h(v) &\cdot \bbf(\pi_h(\bu)) \; d\bx  +\int_{\partial \Omega}\pi_h(v) \big (\hat{\bbf}_\bn(\pi_h(\bu),\bu_b)-\bbf(\pi_h(\bu))\cdot \bn\big ) \; d\gamma\\&=
-\int_\Omega \nabla \pi_h(v) \cdot \big ( \bbf(\pi_h(\bu))-\bbf(\bu)\big ) \; d\bx
+\int_{\partial \Omega}\pi_h(v) \big ( \hbbf_\bn(\pi_h(\bu),\bu_b)-\hbbf_\bn(\bu,\bu_b)\big )\; d\gamma\\&\qquad \qquad - \int_{\partial \Omega}\pi_h(v)\big( \bbf(\pi_h(\bu))-\bbf(\bu) \big )\cdot \bn \; d\gamma\\
&=\mathcal{O}(h^{k+1})
\end{split}
\end{equation*}
since the flux on the boundary is the upwind flux $\mathcal{F}_\bn$, and using  the approximation properties of $\pi_h(\bu)$.

Then
$$\sum\limits_{e\in \mathcal{E}_h} \int_e[\pi_h(v)]\hbbf_\bn(\pi_h(\bu),\pi_h(\bu)^{-})\; d\gamma=\mathcal{O}(h^{-d+1}) \times \mathcal{O}(h^{d-1})\times  \mathcal{O}(h^{k+1})\times \mathcal{O}(1)=\mathcal{O}(h^{k+1}),$$
$$\sum\limits_{K\subset \Omega}\frac{1}{\#K}\bigg ( \sum\limits_{\sigma,\sigma'\in K} (\bv_\sigma-\bv_{\sigma'})\bigg ( \Phi_\sigma^K(\pi_h(\bu))-\Phi_\sigma^{K, Gal}(\pi_h(\bu)) \bigg ) \bigg )= \mathcal{O}(h^{-d})\times  \mathcal{O}(h) \times  \mathcal{O}(h^{k+d})=\mathcal{O}(h^{k+1}),$$
and similarly
$$\sum\limits_{\Gamma\subset \partial \Omega} \frac{1}{\#\Gamma}\bigg ( \sum\limits_{\sigma,\sigma'\in \Gamma} (\bv_\sigma-\bv_{\sigma'})(\Phi_\sigma^\Gamma \big (\pi_h(\bu),\bu_b)-\Phi_\sigma^{Gal,\Gamma}(\pi_h(\bu),\bu_b)\big )\bigg )=\mathcal{O}(h^{-d+1})\times\mathcal{O}(h) \times  \mathcal{O}(h^{k+d-1})=\mathcal{O}(h^{k+1})$$
thanks to the regularity of the mesh, that $\pi_h(v)$ is the interpolant of a $C^1$ function and the previous estimates.
\end{proof}
\begin{remark}[Numerical integration]
In practice, the integrals are evaluated by numerical integration.  The results still holds true provided the quadrature formula are of order $k+1$. This is in contrast with the common practice, but let us emphasis this is valid only for steady problems. However, similar arguments can be developed for unsteady problems, see \cite{Mario,Abgrall2017}. 
\end{remark}

%% file: flux.tex
\section{Flux formulation of Residual Distribution schemes}\label{flux}
%%%%%%%
In this section we show that the scheme \eqref{RD:scheme} also admits a flux formulation, with an explicit form of the flux: the method is also locally conservative. Local conservation is of course  well known for the Finite Volume and discontinuous Galerkin  approximations. It is much less understood for the  continuous finite elements methods, despite the papers \cite{Hughes1,BurQS:10}.
Referring to \eqref{flux:def}, the aim of this section  is to define $\hbbf$ and
$S^\pm$ in the RDS case.

We first show why a finite volume can be reinterpreted as an RD scheme.  This helps to understand the structure of the problem. Then we show that any RD scheme can be equivalently rephrased as a finite volume scheme, we explicitly provide the flux formula as well as the control  volumes. In order to illustrate this result, we give several examples: the general RD scheme with $\PP^1$ and $\PP^2$ approximation on simplex, the case of a $\PP^1$ RD scheme using a particular form of the residuals so that one can better see the connection with more standard formulations, and finally  an example with a discontinuous Galerkin formulation using $\PP^1$  approximation.
%We first provide explicit formula for the $\PP^1$ and then $\PP^2$ case. This enables to point out some generalisation since the original framework is using a globally continuous approximation of the solution. Then we generalise all this, forgetting about the shape of the element $K$, the degree and the type of methods: the approximation can now be only locally continuous as for dG.

%%%%%%%%%%%%%%%%%%%%%%%%%%%%%%%%%%%%%%%%

  %%%%%%%%%%%%%%%%%%%%%%%%%%%%
\subsection{Finite volume as Residual distribution schemes}
Here, we rephrase \cite{Abgrall99}. The notations are defined in Figure \ref{fig:fv}.
\begin{figure}[h]
\begin{center}
\subfigure[]{\includegraphics[width=0.45\textwidth]{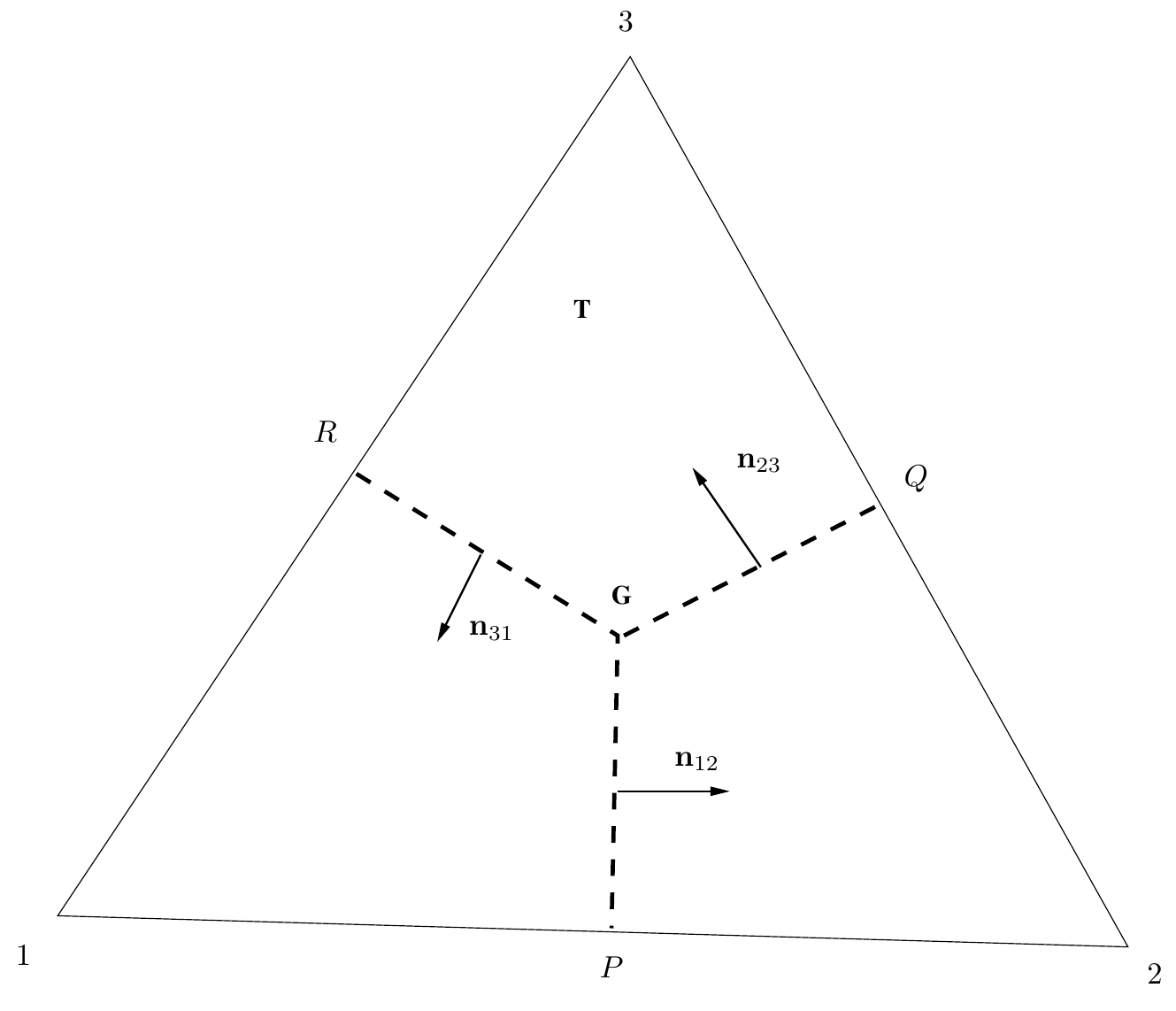}}
\subfigure[]{\includegraphics[width=0.45\textwidth]{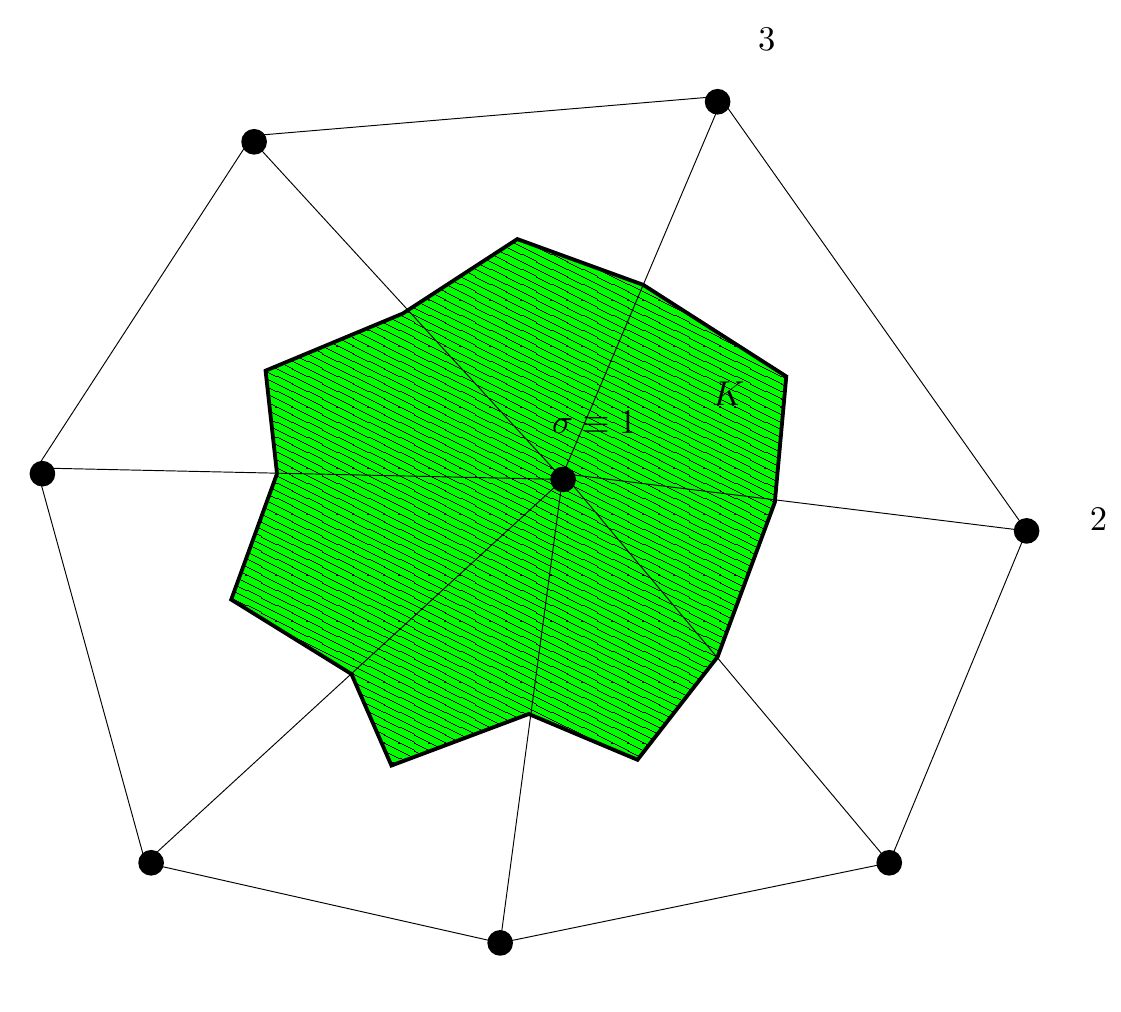}}
\end{center}
\caption{\label{fig:fv} Notations for the finite volume schemes. On the left: definition of the control volume for the degree of freedom $\sigma$.
 The vertex $\sigma$ plays the role of the vertex $1$ on the left picture for the triangle K. The control volume $C_\sigma$ associated to $\sigma=1$ is green on the right and corresponds to $1PGR$ on the left. The vectors $\bn_{ij}$ are normal to the internal edges scaled by the corresponding edge length}
\end{figure}
Again, we specialize ourselves to the case of triangular elements, but  \emph{exactly the same arguments} can be given for more general elements,
 provided a conformal approximation space can be constructed. This is  the case for
triangle elements, and we can take $k=1$.

The control volumes in this case are defined as the median cell, see figure \ref{fig:fv}.
We concentrate on the  approximation of $\text{div } \bbf$, see equation \eqref{eq1}.   Since the boundary of $C_\sigma$ is a closed polygon, the scaled outward normals $\bn_\gamma$ to $\partial C_\sigma$ sum up to 0:
$$
\sum_{\gamma \subset \partial C_\sigma}\bn_\gamma=0$$
where $\gamma$ is any of the segment included in $\partial C_\sigma$, such as $PG$ on Figure \ref{fig:fv}. 
Hence
\begin{equation*}
\begin{split}
\sum_{\gamma \subset \partial C_\sigma} \hbbf_{\bn_\gamma }(\bu_\sigma , \bu^-& )= \sum_{\gamma \subset \partial C_\sigma} \hbbf_{\bn_\gamma }(\bu_\sigma, \bu^- )- \bigg (\sum_{\gamma \subset \partial C_\sigma}\bn_\gamma\bigg )\cdot \bbf (\bu_\sigma)\\
&=\sum\limits_{K, \sigma\in K} \sum\limits_{\gamma \subset \partial C_\sigma\cap K} \big ( \hbbf_{\bn_\gamma }(\bu_\sigma, \bu^- )-\bbf (\bu_\sigma)\cdot \bn_\gamma \big )
\end{split}
\end{equation*}
To make things explicit, in $K$, the internal boundaries are $PG$, $QG$ and $RG$, and those around $\sigma\equiv 1$ are $PG$ and $RG$.
We set
\begin{equation}
\begin{split}
\Phi_\sigma^K(\bu^h)&=\sum\limits_{\gamma\subset \partial C_\sigma\cap K} \big ( \hbbf_{\bn_\gamma }(\bu_\sigma, \bu^- )-\bbf (\bu_\sigma)\cdot \bn_\gamma \big )\\
&=\sum\limits_{\gamma\subset \partial ( C_\sigma\cap K )}  \hbbf_{\bn_\gamma }(\bu_\sigma, \bu^- ).
\end{split}
\label{fv:res:sigma}
\end{equation}
The last relation uses the consistency of the flux and the fact that $C_\sigma\cap K$ is a closed polygon. The quantity $\Phi_\sigma^K(\bu^h)$ is the normal flux on $C_\sigma\cap K$.
If now we sum up these three quantities and get:
\begin{equation*}
\begin{split}
\sum_{\sigma\in K} \Phi_\sigma^K(\bu_h)&= \bigg ( \hbbf_{\bn_{12}}(\bu_1,\bu_2)-\hbbf_{\bn_{13}}(\bu_1,\bu_3)-\bbf(\bu_1)\cdot\bn_{12}+\bbf(\bu_1)\cdot\bn_{31}\bigg )\\
&+\bigg ( \hbbf_{\bn_{23}}(\bu_2,\bu_3)-\hbbf_{\bn_{12}}(\bu_2,\bu_1)+\bbf(\bu_2)\cdot\bn_{12}-\bbf(\bu_2)\cdot\bn_{23}\bigg )\\
&+\bigg ( -\hbbf_{\bn_{23}}(\bu_3,\bu_2)+\hbbf_{\bn_{31}}(\bu_3,\bu_1)-\bbf(\bu_3)\cdot\bn_{23}+\bbf(\bu_3)\cdot\bn_{31}\bigg )\\
&= \bbf(\bu_1)\cdot \big ( \bn_{12}-\bn_{31}\big ) +\bbf(\bu_2)\cdot \big ( -\bn_{23}+\bn_{31}\big )
+\bbf(\bu_3)\cdot \big ( \bn_{31}-\bn_{23}\big )\\
&=\bbf(\bu_1)\cdot\frac{\bn_1}{2}+\bbf(\bu_2)\cdot\frac{\bn_2}{2}+\bbf(\bu_3)\cdot\frac{\bn_3}{2}
\end{split}
\end{equation*}
where $\bn_j$ is the scaled inward normal of the edge opposite to vertex $\sigma_j$, i.e. twice the gradient of the $\PP^1$ basis function
 $\varphi_{\sigma_j}$ associated to this degree of freedom.
Thus, we can reinterpret the sum as the boundary integral of the Lagrange interpolant of the flux.
The finite volume scheme is then a residual distribution scheme with residual defined by \eqref{fv:res:sigma}
and a total residual defined by
\begin{equation}
\label{fv:tot:residu}
\Phi^K:=\int_{\partial K} \bbf^h\cdot \bn , \qquad \bbf^h=\sum_{\sigma\in K} \bbf(\bu_\sigma)\varphi_\sigma.
\end{equation}

\subsection{Residual distribution schemes as  finite volume schemes.}
In this section, we show how to interpret  RD schemes as finite volume schemes. This amounts to defining control volumes and flux functions. 
We first have to adapt the notion of consistency. As recalled in the section \ref{sec:conservation}, two of the key arguments in the proof  of the Lax-Wendroff theorem are related to the structure of the flux, for classical finite volume schemes. In \cite{AbgrallRoe}, the proof is adapted to the case of Residual Distribution schemes. The property that stands for the consistency is that if all the states  are identical in an element, then each of the residuals vanishes. Hence, we define a multidimensional flux as follows:
\begin{definition}\label{MD:consistency}
 A multidimensional flux 
$$\hbbf_\bn:=\hbbf_\bn(\bu_1, \ldots , \bu_N)$$
is consistent if, when $\bu_1= \bu_2= \ldots = \bu_N=\bu$ then
$$\hbbf_\bn(\bu, \ldots , \bu)=\bbf(\bu)\cdot \bn.$$
\end{definition}
We proceed first with the general case and show the connection with elementary fact about graphs, and then provide several examples. The results of this section apply to any 
finite element method but also to discontinuous Galerkin methods. There is no need for exact evaluation of integral formula (surface or boundary), so that these results apply to schemes as they are implemented.

\subsubsection{General case}
One can deal with the general case, i.e when $K$ is a polytope contained in $\R^d$ with degrees of freedoms on the boundary of $K$. The set $\mathcal{S}$ is the set of degrees of freedom. We consider a triangulation $\mathcal{T}_K$ of $K$ whose vertices  are exactly the elements of $\mathcal{S}$. Choosing an orientation of $K$, it is propagated on $\mathcal{T}_K$: the edges are oriented.

The problem is to find quantities $\hbbf_{\sigma,\sigma'}$ for any edge $[\sigma,\sigma']$ of   $\mathcal{T}_K$ such that:
\begin{subequations}\label{GC:1}
\begin{equation}
\label{GC:1.1}
\Phi_\sigma=\sum_{\text{ edges }[\sigma,\sigma']} \hbbf_{\sigma,\sigma'}+\hbbf_\sigma^{b}
\end{equation}
with
\begin{equation}
\hbbf_{\sigma,\sigma'}=-\hbbf_{\sigma',\sigma}
\label{GC:1.2}
\end{equation}
and $\hbbf_\sigma^{b}$ is the 'part' of $\oint_{\partial K} \hbbf_\bn(\bu^h,\bu^{h,-}) \; d\gamma$ associated to $\sigma$. {The control volumes will be defined by their normals so that we get consistency.}

Note that \eqref{GC:1.2} implies the conservation relation
\begin{equation}
\label{GC:conservation}
\sum\limits_{\sigma\in K}\Phi_\sigma=\sum\limits_{\sigma\in K}\hbbf_\sigma^b.
\end{equation}
In short, we will consider 
\begin{equation}
\label{BC:1.3}
\hbbf_\sigma^b=\oint_{\partial K} \varphi_\sigma\; \hbbf_\bn (\bu^h,\bu^{h,-}) \; d\gamma,
\end{equation}
\end{subequations}
but other  examples can be considered provided the consistency \eqref{GC:conservation} relation holds true, see for example section \ref{particular}.
Any edge $[\sigma,\sigma']$ is either direct or, if not, $[\sigma',\sigma]$ is direct. Because of \eqref{GC:1.2}, we only need to know $\hbbf_{\sigma,\sigma'}$ for direct edges. Thus we introduce the notation $\hbbf_{\{\sigma,\sigma'\}}$ for  the flux  assigned to  the direct edge whose extremities are $\sigma$ and $\sigma'$. We can rewrite \eqref{GC:1.1} as, for any $\sigma\in \mathcal{S}$,
\begin{equation}
\label{GC:1.1bis}
\sum_{\sigma'\in \mathcal{S}} \varepsilon_{\sigma,\sigma'} \hbbf_{\{\sigma,\sigma'\}}=\Psi_\sigma:=\Phi_\sigma-\hbbf_\sigma^b,
\end{equation}
with $$
\varepsilon_{\sigma,\sigma'}=\left \{
\begin{array}{ll}
0& \text{ if }\sigma \text{ and }\sigma' \text{ are not on the same edge of }\mathcal{T},\\
1& \text{ if } [\sigma,\sigma']\text{ is an edge and } \sigma \rightarrow \sigma' \text{ is direct,}\\
-1&  \text{ if } [\sigma,\sigma']\text{ is an edge and } \sigma' \rightarrow \sigma \text{ is direct.}
\end{array}
\right .
$$
$\mathcal{E}^+$ represents the set of direct edges.

Hence the problem is to find  a vector $\hbbf=(\hbbf_{\{\sigma,\sigma'\}})_{\{\sigma,\sigma'\} \text{ direct edges}}$ such that
$$A\hbbf=\Psi$$
where $\Psi=(\Psi_\sigma)_{\sigma\in \mathcal{S}}$ and $A_{\sigma \sigma'}=\varepsilon_{\sigma,\sigma'}$.

We have  the following lemma which shows the existence of a solution.
\begin{lemma}\label{lemma:flux}
For any couple $\{\Phi_\sigma\}_{\sigma\in \mathcal{S}}$ and $\{\hbbf_\sigma^{b}\}_{\sigma\in \mathcal{S}}$ satisfying the condition  \eqref{GC:conservation}, there exists numerical flux functions $\hbbf_{\sigma,\sigma'}$ that satisfy \eqref{GC:1}. Recalling that the  matrix of the Laplacian of the graph is $L=AA^T$, we have
\begin{enumerate}
\item The rank of $L$ is $|\mathcal{S}|-1$ and its image is $\big (\text{span}\{\mathbf{1}\})^\bot$. We still denote the inverse of $L$ on $\big (\text{span}\{\mathbf{1}\} )^\bot$ by $L^{-1}$,
\item 
With the previous notations, a solution is 
\begin{equation}
\label{eq:lemma}\big (\hbbf_{\{\sigma,\sigma'\}}\big )_{\{\sigma,\sigma'\} \text{ direct edges}}=A^TL^{-1} \big (\Psi_\sigma\big )_{\sigma\in \mathcal{S}}.\end{equation}
\end{enumerate}
\end{lemma}
\begin{proof}
We first have $\mathbf{1}^T\,A=0$: $\text{ Im } A\subset \big (\text{span }\{1\}\big )^\bot (\subset \R^{|\mathcal{S}|})$. Let us show that we have equality. 
In order to show this, we notice that  the matrix $A$ is nothing more that the incidence matrix of the oriented graph $\mathcal{G}$ defined by the triangulation $\mathcal{T}$.  It is known \cite{graph} that its null space of $L$  is equal to the number of connected  components of the graph, i.e. here $\dim \ker L=1$. Since 
$$ L\,\mathbf{1}=0,$$ we see that $\ker L=\text{span }\{\mathbf{1}\}$, so that $\text{ Im }L= \big (\text{span }\{\mathbf{1}\}\big )^T$ because $L$ is symmetric. We can define the inverse of $L$ on $\text{Im }L$,  denoted by $L^{-1}$. 

Let $x\in \big (\text{span }\{\mathbf{1}\}\big )^\bot=\text{ Im }L$. There exists $y\in \R^{|S|}$ such that $x=Ly=A(A^Ty)$: this shows that $x\in \text{ Im }A$ and thus
$\text{ Im }A=\big (\text{span }\{\mathbf{1}\}\big )^\bot=\big (\text{Im }L\big )^\bot$.  From this we deduce that $\text{rank }A=|\mathcal{S}|-1$ because $\text{ Im } A\subset \R^{|\mathcal{S}|}$.

Let $\Psi\in \R^{|\mathcal{S}|}$ be such that $\langle \mathbf{1},\Psi\rangle=0$. We know there exists a unique $z\in \big (\text{span }\{\mathbf{1}\}\big )^\bot$ such that $Lz=\Psi$, i.e.
$$A(A^Tz)=\Psi.$$
This shows that a solution is given by \eqref{eq:lemma}.
\end{proof}

This set of flux are consistent and we can estimate the normals $\bn_{\sigma,\sigma'}$.
In the case of a constant state, we have $\Phi_\sigma=0$ for all $\sigma\in K$. Let us assume that
\begin{equation}
\label{GC:consistency}
\hbbf_\sigma^b=\bbf(\bu^h)\cdot \mathbf{N}_\sigma
\end{equation}
with $\sum\limits_{\sigma\in K} \mathbf{N}_\sigma=0$: this is the case for all the examples we consider. The flux $\bbf(\bu^h)$ has components on the canonical basis of $\R^d$:
$\bbf(\bu^h)=\big (f_1(\bu^h), \ldots , f_d(\bu^h)\big )$, so that
$$\hbbf_\sigma^b=\sum\limits_{i=1}^d f_i(\bu^h)\mathbf{N}^i_\sigma.$$
%If $\lambda_2, \ldots , \lambda_{\#K}$ are the non zero eigenvalues of $L$ and $(\mathbf{e}_2, \ldots , \mathbf{e}_{\#K})$ the corresponding eigenvectors, we have for any $\bx$,
%$$L(\bx)=\sum_{l=2}^{\#K} \lambda_l ( \bx\cdot \mathbf{e}_j) \mathbf{e}_j$$ so that
%$$L^{-1}(\bx)=\sum_{l=2}^{\#K}\frac{1}{\lambda_i}( \bx\cdot e_j) e_j.$$
Applying this to $\big (\hbbf_{\sigma_1}^b, \ldots, \hbbf_{\sigma_{\#K}}^b\big )$, we see that the $j$-th component of $\bn_{\sigma,\sigma'}$ for $[\sigma,\sigma']$ direct, must satisfy:
$$\text{ for any }\sigma\in K, \; \mathbf{N}^j_\sigma=\sum\limits_{[\sigma,\sigma']\text{ edge }}\varepsilon_{\sigma,\sigma'}\bn_{\sigma,\sigma'}^j$$
i.e.
$$\big ( \mathbf{N}^j_{\sigma_1}, \ldots , \mathbf{N}^j_{\sigma_{\#K}}\big )^T= A \; \big ( \bn_{\sigma,\sigma'}^j\big )_{[\sigma,\sigma']\in \mathcal{E}^+}.$$
We can solve the system and the solution, with some abuse of language, is
\begin{equation}
\label{GC:normals}
\big ( \bn_{\sigma,\sigma'}\big )_{[\sigma,\sigma']\in \mathcal{E}^+}=A^TL^{-1} \big ( \mathbf{N}_{\sigma_1}, \ldots , \mathbf{N}_{\sigma_{\#K}}\big )^T
\end{equation}
This also defines the control volumes since we know their normals. We can state:
\begin{proposition}
If the residuals $(\Phi_\sigma)_{\sigma\in K}$ and the boundary fluxes $(\hbbf_\sigma^b)_{\sigma\in K}$ satisfy \eqref{GC:conservation}, and if the boundary fluxes satisfy the consistency relations \eqref{GC:consistency}, then we can find  a set of consistent flux $(\hbbf_{\sigma,\sigma'})_{[\sigma,\sigma']} $ satisfying \eqref{GC:1}. They are given by \eqref{eq:lemma}. In addition, for a constant state,
$$\hbbf_{\sigma,\sigma'}(\bu^h)=\bbf(\bu^h)\cdot\bn_{\sigma,\sigma'}$$ for the normals defined by \eqref{GC:normals}.
\end{proposition}

\bigskip

We can state a couple of general remarks:
\begin{remark}
\begin{enumerate}
\item The flux $\hbbf_{\sigma,\sigma'}$ depend on the $\Psi_\sigma$ and not directly on the $\hbbf_\sigma^b$. We can design the fluxes independently of the boundary flux, and their consistency directly comes from the consistency of the boundary fluxes.
\item 
The residuals depends on more than 2 arguments. For stabilized finite element methods, or the non linear stable residual distribution
 schemes, see e.g.  \cite{Hughes1,struijs,abgrallLarat}, the residuals depend on all the states on  $K$. Thus
the formula \eqref{eq:lemma} shows that the flux depends on more than two states in contrast to the  1D case. In the finite volume case however, the support of the flux function is generally larger than the three states of $K$, think for example of an ENO/WENO method, or a simpler MUSCL one.
\item The formula \eqref{eq:lemma} are influenced by the form of the total residual \eqref{fv:tot:residu}.  We show in the next paragraph how this can be generalized.
%\item We have set at the beginning that $\hbbf_{{\sigma,\sigma'}}=-\hbbf_{{\sigma',\sigma}}$. The formula \eqref{eq:lemma} are antisymmetric with respect to the indices, and then do respect the assumed equality.
\item The formula \eqref{eq:lemma} make no assumption on the approximation space $V^h$: they are valid for continuous and discontinuous approximations. The structure of the approximation space appears only in the total residual.
\end{enumerate}

\end{remark}

\subsubsection{Some particular cases: fully explicit formula}\label{particular}
Let $K$ be a fixed triangle. We are given a set of residues $\{\Phi_\sigma^K\}_{\sigma\in K}$, our aim here is to define a
 flux function such that relations similar to \eqref{fv:res:sigma} hold true. We explicitly give the formula 
for $\PP^1$ and $\PP^2$ interpolant.

\paragraph{The general  $\PP^1$ case.}
The adjacent matrix is 
$$A=\left (\begin{array}{rrr} 1&0&-1\\
-1&1&0\\
0&-1&1
\end{array}\right ).
$$
A straightforward calculation shows that the matrix $L=A^TA$ has eigenvalues $0$ and $3$ with multiplicity 2 with eigenvectors
$$R=\begin{pmatrix}
\frac{1}{\sqrt{3}} & \frac{1}{\sqrt{2}} & \frac{1}{\sqrt{6}}\\
\frac{1}{\sqrt{3}} & \frac{-1}{\sqrt{2}} & \frac{1}{\sqrt{6}}\\
\frac{1}{\sqrt{3}} & 0                         & \frac{-2}{\sqrt{6}}
\end{pmatrix}
$$
To solve $A\hbbf=\Psi$, we decompose $\Psi$ on the eigenbasis:
$$\Psi=\alpha_2 R_2+\alpha_3R_3$$ where explicitly
$$\begin{array}{l}
\alpha_2=\frac{1}{\sqrt{2}} \big (\Psi_1-2\Psi_2+\Psi_3\big )\\
\\
\alpha_3=\sqrt{\frac{3}{2}}\big (\Psi_1-\Psi_3\big )
\end{array}
$$
so that 
$$\hbbf=\frac{1}{3}\begin{pmatrix}\Psi_1-\Psi_3 \\ \Psi_2-\Psi_3 \\ \Psi_3-\Psi_2 \end{pmatrix}.$$

In order to describe the control volumes, we first have to make precise the normals $\bn_\sigma$ in that case. It is easy to see that in all the cases described above, we have 
$$\normal_\sigma=-\frac{\bn_\sigma}{2}.$$ Then a short calculation shows that
$$\begin{pmatrix}
\bn_{12} \\ \bn_{23} \\ \bn_{31} \end{pmatrix}=
\frac{1}{6}\begin{pmatrix} \bn_1-\bn_2 \\ \bn_2 -\bn_3 \\ \bn_3 -\bn_1 \end{pmatrix}.
$$
Using elementary geometry of the triangle, we see that these  are the normals of the elements of the dual mesh. For example, the normal $\bn_{12}$ is the normal of 
 $PG$, see figure \ref{fig:fv}.
 
 Relying more on the geometrical interpretation (once we know the control volumes), we can recover the same formula by elementary calculations, see \cite{icm}.
%%%%%%%%%%%%%

\paragraph{The general example of the $\PP^2$ approximation. }
Using a similar method, we get (see figure \ref{fig:P2} for some notations):
$$
\begin{array}{lcl}
\hbbf_{14}&=&\dfrac{1}{12}\big (\Psi_1-\Psi_4\big )+\dfrac{1}{36}\big ( \Psi_6-\Psi_5\big )+\dfrac {7}{36}\big (\Psi_1- \Psi_2\big )+\dfrac {5 }{36}\big (\Psi_3-\Psi_1\big )\\
&\\
\hbbf_{16}&=&\dfrac{1}{12}\big ( \Psi_4-\Psi_1\big )+\dfrac {5}{36}\big ( \Psi_5-\Psi_1)
+\dfrac {7}{36}\big ( \Psi_6-\Psi_1\big ) +\dfrac{1}{36}\big ( \Psi_3- \Psi_2\big )\\
&\\
\hbbf_{46}&=&\dfrac{2}{9}\big (\Psi_2-\Psi_6\big )+\dfrac{1}{9}\big (  \Psi_3- \Psi_5\big )\\
&\\
\hbbf_{54}&=&\dfrac{2}{9}\big (  \Psi_5-\Psi_2\big )+\dfrac{1}{9}\big ( \Psi_5-\Psi_1\big )\\
\end{array}
$$
$$
\begin{array}{lcl}
%&\\
\hbbf_{42}&=&\dfrac {7}{36}\big (\Psi_2-\Psi_3\big ) +\dfrac{5}{36}\big (\Psi_1-\Psi_3\big )+\dfrac{1}{12}\big(\Psi_6-\Psi_3\big )+\dfrac{1}{36}\big (\Psi_5-\Psi_4\big ) \\

&\\
\hbbf_{25}&=&\dfrac{1}{36}\big (\Psi_2-\Psi_1\big )+\dfrac{5}{36}\big (\Psi_3-\Psi_5\big ) +\dfrac{7}{36}\big ( \Psi_3-\Psi_5\big )+\dfrac{1}{12}\big (\Psi_3-\Psi_6\big )  \\

&\\
\hbbf_{53}&=&\dfrac{1}{36}\big (\Psi_1-\Psi_6\big )+\dfrac{5}{36}\big (\Psi_3-\Psi_5\big )+\dfrac{7}{36}\big (\Psi_4-\Psi_5\big )+\dfrac{1}{12}\big (\Psi_2-\Psi_5\big )
\\
&\\
\hbbf_{63}&=& \dfrac{1}{36}\big (\Psi_4-\Psi_3\big )+\dfrac{5}{36}\big (\Psi_5-\Psi_1\big )+\dfrac{7}{36}\big (\Psi_5-\Psi_6\big )+\dfrac{1}{12}\big (\Psi_5-\Psi_2\big )\\
&\\

\hbbf_{65}&=&\dfrac{1}{9}\big (\Psi_1- \Psi_3\big )+\dfrac{2}{9}\big ( \Psi_6- \Psi_4\big )\end {array}
$$
Then we choose the boundary flux:
$$\hbbf_\sigma^b=\int_{\partial K}\varphi_\sigma\bn\; d\gamma$$ and get:
$$
\begin{array}{lll}
\normal_l=-\dfrac{\bn_l}{6} & \text{if }l=1,2,3\\ &&\\
\normal_4=\dfrac{\bn_3}{3}& \normal_5=\dfrac{\bn_1}{3}& \normal_6=\dfrac{\bn_2}{3}
\end{array}
$$
The normals are given by:
$$
\begin{array}{lcl}
\bn_{14}&=&\dfrac{1}{12}\big (\normal_1-\normal_4\big )+\dfrac{1}{36}\big ( \normal_6-\normal_5\big )+\dfrac {7}{36}\big (\normal_1- \normal_2\big )+\dfrac {5 }{36}\big (\normal_3-\normal_1\big )\\
&\\
\bn_{16}&=&\dfrac{1}{12}\big ( \normal_4-\normal_1\big )+\dfrac {5}{36}\big ( \normal_5-\normal_1)
+\dfrac {7}{36}\big ( \normal_6-\normal_1\big ) +\dfrac{1}{36}\big ( \normal_3- \normal_2\big )\\
&\\

\bn_{46}&=&\dfrac{2}{9}\big (\normal_2-\normal_6\big )+\dfrac{1}{9}\big (  \normal_3- \normal_5\big )\\
&\\
\bn_{54}&=&\dfrac{2}{9}\big (  \normal_5-\normal_2\big )+\dfrac{1}{9}\big ( \normal_5-\normal_1\big )
\end{array}
$$
$$
\begin{array}{lcl}
%&\\
\bn_{42}&=&\dfrac {7}{36}\big (\normal_2-\normal_3\big ) +\dfrac{5}{36}\big (\normal_1-\normal_3\big )+\dfrac{1}{12}\big(\normal_6-\normal_3\big )+\dfrac{1}{36}\big (\normal_5-\normal_4\big ) \\

&\\
\bn_{25}&=&\dfrac{1}{36}\big (\normal_2-\normal_1\big )+\dfrac{5}{36}\big (\normal_3-\normal_5\big ) +\dfrac{7}{36}\big ( \normal_3-\normal_5\big )+\dfrac{1}{12}\big (\normal_3-\normal_6\big )  \\

&\\
\bn_{53}&=&\dfrac{1}{36}\big (\normal_1-\normal_6\big )+\dfrac{5}{36}\big (\normal_3-\normal_5\big )+\dfrac{7}{36}\big (\normal_4-\normal_5\big )+\dfrac{1}{12}\big (\normal_2-\normal_5\big )
\\
&\\
\bn_{63}&=& \dfrac{1}{36}\big (\normal_4-\normal_3\big )+\dfrac{5}{36}\big (\normal_5-\normal_1\big )+\dfrac{7}{36}\big (\normal_5-\normal_6\big )+\dfrac{1}{12}\big (\normal_5-\normal_2\big )\\
&\\

\bn_{65}&=&\dfrac{1}{9}\big (\normal_1- \normal_3\big )+\dfrac{2}{9}\big ( \normal_6- \normal_4\big )\end {array}
$$

\bigskip

There is not uniqueness, and it is possible to construct different solutions to the problem. In what follows, we show another possible construction. 
We consider the set-up defined by Figure \ref{fig:P2}.
\begin{figure}[h!]
\begin{center}
\includegraphics[width=0.45\textwidth]{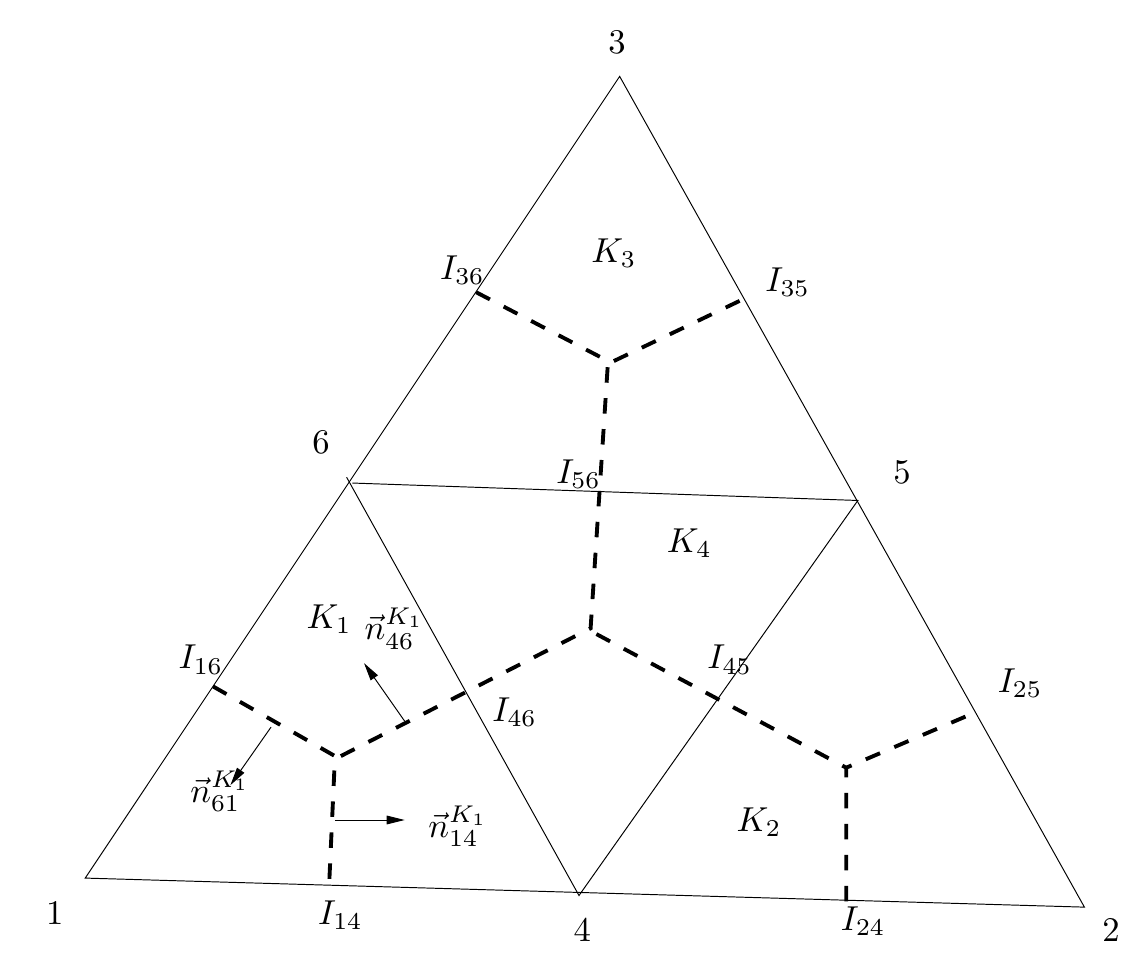}
\end{center}
\caption{\label{fig:P2} Geometrical elements for the $\PP^2$ case. $I_{ij}$ is the mid-point between the vertices $i$ and $j$. The intersections of the dotted lines are the centroids of the sub-elements.}
\end{figure}
The triangle is split first into 4 sub-triangles $K_1$, $K_2$, $K_3$ and $K_4$. From this sub-triangulation,
 we can construct a dual mesh as in the $\PP^1$ case and we have represented the  6 sub-zones
 that are the intersection of the dual control volumes and the triangle $K$.  Our notations are as follow: given any sub-triangle $K_\xi$,
 if $\gamma_{ij}$ is intersection between two adjacent control volumes (associated to $\sigma_i$ and $\sigma_j$ vertices of $K_\xi$), the normal
 to $\gamma_{ij}$ in the direction $\sigma_i$ to $\sigma_j$ is denoted by $\bn_{ij}^{\xi}$. Similarly the flux across $\gamma_{ij}$ is denoted 
$\hbbf_{ij}^\xi$.

Then we need to define boundary fluxes.  If $\sigma$ belongs to $K_l$, we denote the boundary flux as $\hbbf_\sigma^{K_l}$. A rather  natural condition is that
$$\begin{array}{ll}
\hbbf_l^{K_l}=\hbbf_l^b & l=1,2,3\\
\hbbf_4^{K_l}=\frac{1}{3}\hbbf_4^b & l=1,2,4\\
\hbbf_5^{K_l}=\frac{1}{3}\hbbf_5^b & l=2,3,4\\
\hbbf_6^{K_l}=\frac{1}{3}\hbbf_6^b & l=1,3,4.\\
\end{array}
$$
We recover the conservation relation. Other choices are possible since this one is arbitrary: the only true condition is that the sum of the boundary flux is equal to the sum  of the $\hbbf_j^b$ for $j=1, \ldots, 6$: this is the conservation relation.

Then we set:
\begin{equation}
\label{rd:fv:P2}
\begin{array}{lll} 
\Phi_1&= -\hbbf_{\bn_{61}}^1+\hbbf_{\bn_{14}}^1&+\hbbf_1^b\\
\Phi_2&=-\hbbf_{\bn_{42}}^2+\hbbf_{\bn_{25}}^2&+\hbbf_2^b\\
\Phi_3&=-\hbbf_{\bn_{53}}^3+\hbbf_{\bn_{36}}^3&+\hbbf_3^b\\
\Phi_4&=-\hbbf_{\bn_{14}}^1+\big (\hbbf_{\bn_{46}}^1-\hbbf_{\bn_{64}}^4\big )+\big (\hbbf_{\bn_{45}}^4-\hbbf_{\bn_{54}}^2\big )+\hbbf_{\bn_{42}}^2&+\hbbf_4^b\\
\Phi_5&= -\hbbf_{\bn_{25}}^2+\big ( \hbbf_{\bn_{54}}^2-\hbbf_{\bn_{45}}^4\big ) + \big ( \hbbf_{\bn_{56}}^4-\hbbf_{\bn_{65}}^3\big ) + \hbbf_{\bn_{53}}^3&+\hbbf_5^b\\
\Phi_6&=-\hbbf_{\bn_{36}}^3+\big (\hbbf_{\bn_{65}}^3-\hbbf_{\bn_{56}}^4\big )+\big (\hbbf_{\bn_{64}}^4-\hbbf_{\bn_{46}}^1\big )+\hbbf_{\bn_{61}}^1&+\hbbf_6^b
\end{array}
\end{equation}
We can group the terms in \eqref{rd:fv:P2}  by sub-triangles, namely:
\begin{equation}
\label{rd:fv:P2:2}
\begin{array}{lclclcl}
\Phi_1&=& \big (-\hbbf_{\bn_{61}}^1+\hbbf_{\bn_{14}}^1+\hbbf_1^b \big )& \\
\Phi_2&=&\big (-\hbbf_{\bn_{42}}^2+\hbbf_{\bn_{25}}^2+\hbbf_2^b \big )&\\
\Phi_3&=&\big (-\hbbf_{\bn_{53}}^3+\hbbf_{\bn_{36}}^3+\hbbf_3^b\big )&\\
\Phi_4&=&\big (-\hbbf_{\bn_{14}}^1+\hbbf_{\bn_{46}}^1+\hbbf_4^{K_1}\big ) &+&\big ( -\hbbf_{\bn_{64}}^4+\hbbf_{\bn_{45}}^4+\hbbf_1^{K_4} \big )\\
&&&+&\big (-\hbbf_{\bn_{54}}^2+\hbbf_{\bn_{42}}^2+\hbbf_4^{K_2} \big )\\
\Phi_5&=&\big (-\hbbf_{\bn_{25}}^2+\hbbf_{\bn_{54}}^2+\hbbf_5^{K_2} \big)&
+&\big(-\hbbf_{\bn_{45}}^4+\hbbf_{\bn_{56}}^4+\hbbf_5^{K_4}\big )\\
&&& +&\big (-\hbbf_{\bn_{65}}^3+\hbbf_{\bn_{53}}^3+\hbbf_5^{K_3}\big )\\
\Phi_6&=&\big (-\hbbf_{\bn_{36}}^3+\hbbf_{\bn_{65}}^3+\hbbf_6^{K_3} \big )& 
+&\big (-\hbbf_{\bn_{56}}^4+\hbbf_{\bn_{64}}^4+\hbbf_6^{K_4} \big )\\
&&&+&\big ( -\hbbf_{\bn_{46}}^1+\hbbf_{\bn_{61}}^1+\hbbf_6^{K_1}\big ).
\end{array}
\end{equation}

Then we define the sub-residuals per sub elements:
\begin{equation}
\label{rd:fv:P2:3}
\begin{split}
\Phi_1^1=-\hbbf_{\bn_{61}}^1+\hbbf_{\bn_{14}}^1+\hbbf_1^{b_{\phantom{1}}}  &,\qquad\Phi_4^2=-\hbbf_{\bn_{54}}^2+\hbbf_{\bn_{42}}^2+\hbbf_4^{K_2}\\
\Phi_4^1=-\hbbf_{\bn_{14}}^1+\hbbf_{\bn_{46}}^1+\hbbf_4^{K_1} &,\qquad\Phi_2^2=-\hbbf_{\bn_{42}}^2+\hbbf_{\bn_{25}}^2+\hbbf_2^{K_2}\\
\Phi_6^1=-\hbbf_{\bn_{46}}^1+\hbbf_{\bn_{61}}^1+\hbbf_6^{K_1}  &,\qquad\Phi_5^2=-\hbbf_{\bn_{25}}^2+\hbbf_{\bn_{54}}^2+\hbbf_5^{K_2}\\
&\\
\Phi_5^3=-\hbbf_{\bn_{65}}^3+\hbbf_{\bn_{53}}^3+\hbbf_5^{K_3}  &,\qquad\Phi_4^4=-\hbbf_{\bn_{64}}^4+\hbbf_{\bn_{45}}^4+\hbbf_4^{K_4}\\
\Phi_3^3=-\hbbf_{\bn_{36}}^3+\hbbf_{\bn_{65}}^3+\hbbf_3^{K_3}  &,\qquad\Phi_5^4=-\hbbf_{\bn_{45}}^4+\hbbf_{\bn_{56}}^4+\hbbf_5^{K_4}\\
\Phi_6^3=-\hbbf_{\bn_{36}}^3+\hbbf_{\bn_{65}}^3+\hbbf_6^{K_3}  &,\qquad\Phi_6^4=-\hbbf_{\bn_{56}}^4+\hbbf_{\bn_{64}}^4+\hbbf_6^{K_4},
\end{split}
\end{equation}
so we are back to the $\PP^1$ case: in each sub-triangle, we can define flux that will depend on the 6 states of the element via the boundary flux.
This is legitimate because in the $\PP^1$ case,  we have not used the fact that the interpolation is linear, we have only used the fact that we have 3 vertices. Clearly the fluxes are consistent in the sense of definition \ref{MD:consistency}.

The same argument can be clearly extended to higher degree element, as well as to non triangular element: what is needed is to subdivide the element into 
sub-triangles. %Similarly, this can be extended to discontinuous elements in the spirit of remark \ref{remark:general}.

The two solutions we have presented for the $\PP^2$ case are different: the control volumes are different, since they have more sides in the second case than in the first one.

%% file: examples.tex
\subsubsection{More specific examples}
In what follow we look at the flux form on specific numerical schemes: an extension of the Rusanov scheme, what is called the N scheme after P.L. Roe and a discontinuous Galerkin method.

\medskip
\paragraph{Rusanov residual.}
Here we assume a global continuous approximation.
Assuming that the total residual is evaluated using the Lagrange interpolation of the flux, $\bbf^h=\sum\limits_{\sigma'\in K} \bbf(\bu_\sigma)\varphi_\sigma$, 
we define (the integrals can be evaluated exactly in that case)
\begin{equation}
\label{rusanov}
\Phi_\sigma(\bu^h)=\int_{\partial K}\varphi_\sigma\bbf^h\cdot \bn\; d\gamma-\int_K\nabla \varphi_\sigma \cdot \bbf^h\; d\bx 
+\alpha(\bu_\sigma-\bar\bu), \qquad \bar\bu=\dfrac{\sum\limits_{\sigma'\in K} \bu_{\sigma'}}{\#K}
\end{equation} where $\#K$ is the number of degrees of freedom in $K$ and  $\alpha$ is a parameter that will become explicit later. 

Since  $0=\int_K \varphi_\sigma \; \text{ div }(1)\; d\bx=-\int_K \nabla \varphi_\sigma \; d\bx+\int_{\partial K} \varphi_\sigma \bn \; d\gamma$ and $\sum\limits_{\sigma'\in K}\varphi_{\sigma'}=1$,  we have
\begin{equation*}
\begin{split}
\Phi_\sigma(\bu^h)&=\sum\limits_{\sigma'\in K}\bbf(\bu_{\sigma'})\cdot\bigg ( -\int_K\varphi_{\sigma'} \;\nabla \varphi_\sigma \; d\bx+\int_{\partial K} \varphi_\sigma\varphi_{\sigma'}\bn\; d\gamma \bigg )+\alpha(\bu_\sigma-\bar\bu)\\
&=\sum\limits_{\sigma'\in K}\big ( \bbf(\bu_{\sigma'})-\bbf(\bu_\sigma)\big )\cdot\bigg ( -\int_K\varphi_{\sigma'} \;\nabla \varphi_\sigma \; d\bx+\int_{\partial K} \varphi_\sigma\varphi_{\sigma'}\bn\; d\gamma \bigg )+\alpha(\bu_\sigma-\bar\bu)\\
&=\sum\limits_{\sigma'\in K}\bigg ( \big ( \bbf(\bu_{\sigma'})-\bbf(\bu_\sigma)\big )\cdot\bigg ( -\int_K\varphi_{\sigma'} \;\nabla \varphi_\sigma \; d\bx+\int_{\partial K} \varphi_\sigma\varphi_{\sigma'}\bn\; d\gamma \bigg )+\frac{\alpha}{\#K} \big (\bbu_\sigma-\bbu_{\sigma'}\big ) \bigg )\\
& =\sum\limits_{\sigma'\in K} c_{\sigma\sigma'} (\bbu_\sigma-\bbu_{\sigma'})
\end{split}
\end{equation*}
with
$$c_{\sigma\sigma'}= -\dfrac{\bbf(\bu_{\sigma})-\bbf(\bu_{\sigma'})}{\bu_{\sigma}-\bu_{\sigma'}}\cdot \bigg ( -\int_K\varphi_{\sigma'} \;\nabla \varphi_\sigma \; d\bx+\int_{\partial K} \varphi_\sigma\varphi_{\sigma'}\bn\; d\gamma \bigg ) +\dfrac{\alpha}{\#K}.$$

A local maximum principle is obtained if for any element, and any couple of degrees of freedom in that element,  we have $c_{\sigma \sigma'} \geq 0$.
In the present case, we take
$$\alpha\geq \#K \; \max\limits_{\sigma, \sigma'\in K} \bigg |\dfrac{\bbf(\bu_{\sigma})-\bbf(\bu_{\sigma'})}{\bu_{\sigma}-\bu_{\sigma'}}\cdot \bigg ( -\int_K\varphi_{\sigma'} \;\nabla \varphi_\sigma \; d\bx+\int_{\partial K} \varphi_\sigma\varphi_{\sigma'}\bn\; d\gamma \bigg ) \bigg |.$$

In the case of triangular elements with $\PP^1$ approximation, we have
$$\hbbf_{\sigma\sigma'}=\frac{1}{2}\big ( \int_K\nabla\big (\varphi_\sigma-\varphi_{\sigma'}\big )\cdot \Bf^h\; d\gamma\big ) +\alpha (\bu_\sigma-\bu_{\sigma'}).$$
Using simple geometry (see figure \ref{fig:fv}-a), we get
\begin{equation}
\label{flux:lxf}
\hbbf_{\sigma\sigma'}=\frac{1}{|K|}\big (\int_K\Bf^h\; d\bx\big )\cdot \bn_{\sigma\sigma'}+\alpha (\bu_\sigma-\bu_{\sigma'}).
\end{equation}

We see that this flux is not exactly the classical Rusanov flux
$$\hbbf^{Rus}_{\sigma\sigma'} =\frac{1}{2}\big ( \Bf_\sigma+\Bf_{\sigma'})\cdot \bn_{\sigma\sigma'}+\alpha (\bu_\sigma-\bu_{\sigma'}),$$
but is formally very close to it: it is the sum of a centered part (the surface integral) and a dissipation. We also note that the flux \eqref{flux:lxf} is not necessarily monotone, but 
it is monotone combined with the flux  $\hbbf_\bn^b$.

%\begin{remark}\label{beta}
%The coefficients $\beta_\sigma$ introduced in the relations \eqref{schema RDS SUPG} and \eqref{schema RDS jump} are defined by:
%\begin{equation}\label{eqbeta}
%\beta_\sigma=\dfrac{\max(0,\frac{\Phi_\sigma}{\Phi})}
%{
%\sum\limits_{\sigma'\in K} \max(0,\frac{\Phi_{\sigma'}}{\Phi})}.
%\end{equation} These coefficients are always defined and   guaranty a local maximum principle for \eqref{schema RDS SUPG} and \eqref{schema RDS jump}: this is again a consequence of the conservation properties,  see e.g. \cite{abgrallLarat}.
%\end{remark}
\paragraph{The N scheme.}
Considering the problem \eqref{eq1} with triangular elements. We assume  the existence of an average vector $\overline{\nabla_\bu\Bf}$ such that
$$\frac{1}{2}\sum_\sigma \Bf_\sigma\cdot \bn_\sigma=|K| \; \overline{\nabla_\bu\Bf}\cdot \nabla \bu^h.$$
Here, again both $\Bf$ and $\bu$ are approximated by a linear Lagrange interpolant. 

A simple example of such situation is given by the Burgers problem where $\Bf(\bu)=(\frac{\bu^2}{2}, \bu)^T$. Here
$$\overline{\nabla_\bu\Bf}= \big ( \bar\bu, 1)^T$$
where $\overline{\bu}$ is the arithmetic average of the nodal values. This average is a generalisation of the Roe average \cite{roe1981}, a version for the Euler equations  can be found in \cite{DeconinckRoeStruijs}.

Using this average, the N scheme, see \cite{RoeSidilkover}, can be defined as follows:
\begin{subequations}
\label{Nscheme}
\begin{equation}\label{Nscheme:1}
\Phi_\sigma= k_\sigma^+\big ( \bu_\sigma-\tilde{\bu}\big )
\end{equation}
with
\begin{equation}
\label{Nscheme:2}
k_\sigma=\dfrac{1}{|K|}\int_K \overline{\nabla_\bu\Bf}\cdot\nabla \varphi_\sigma\; d\bx, \qquad k_\sigma^+=\max(k_\sigma,0), k_\sigma^-=\min(k_\sigma,0)
\end{equation}
and 
\begin{equation}
\label{Nscheme:3}
\tilde{\bu}=N\bigg (\sum_{\sigma'\in K} k_{\sigma'}^-\bu_{\sigma'} \bigg ),  \qquad N^{-1}=\sum_{\sigma'\in K} k_{\sigma'}^-
\end{equation}
The value of $N$ is chosen such that the conservation \eqref{conservation:K} holds true.
\end{subequations}
When looking at the flux $\hbbf_{\bn_{\sigma,\sigma'}}$, no particular nice looking structure appears, except in the case of a thin triangle, where the associated flux is a generalisation of Roe's flux.

%In \cite{abgrall-barth} it is shown that the N scheme is also consitent with 
Note that Remark \ref{beta} also applies here, provided that the Rusanov residuals are replaced by those of the N scheme in the definition of $\beta_\sigma$, see \eqref{eqbeta}.

\paragraph{Discontinuous Galerkin schemes ($\PP^1$ case).}
The residual is simply 
$$\Phi_\sigma^K=\oint_{\partial K}\varphi_\sigma\hbbf_\bn(\bu^h,\bu^{h,-})\; d\gamma-\oint_K \nabla\varphi_\sigma\cdot \bbf(\bu^h)\; d\bx.$$

In the $\PP^1$ case, the flux between two DOFs $\sigma$ and $\sigma'$ is given by
$$\hbbf_{\sigma,\sigma'}(\bu^h,\bu^{h,-})=\oint_{\partial K}(\varphi_\sigma-\varphi_{\sigma'})\hbbf_\bn(\bu^h,\bu^{h,-})\; d\gamma-\oint_K \nabla\big (\varphi_\sigma-\varphi_{\sigma'}\big )\cdot \bbf(\bu^h)\; d\bx.$$
Again, from simple geometry, $$\nabla\big (\varphi_\sigma-\varphi_{\sigma'}\big )=-\frac{\bn_{\sigma\sigma'}}{|K|},$$ so that
\begin{equation*}
%\label{flux:DG}
\hbbf_{\sigma,\sigma'}(\bu^h,\bu^{h,-})=\oint_{\partial K}(\varphi_\sigma-\varphi_{\sigma'})\hbbf_\bn(\bu^h,\bu^{h,-})\; d\gamma+\frac{\oint_K  \bbf(\bu^h)\; d\bx}{|K|}\cdot\bn_{\sigma\sigma'}.
\end{equation*}
Note that $\oint_{\partial K}(\varphi_\sigma-\varphi_{\sigma'})\; d\gamma=0$ if we take the same quadrature formula on each edge, as it is usually done. Hence, denoting by $\bar \bu$ the cell average of $\bu^h$ on $K$, we can rewrite the flux as 
\begin{equation}
\label{flux:DG}
\hbbf_{\sigma,\sigma'}(\bu^h,\bu^{h,-})=\frac{\oint_K  \bbf(\bu^h)\; d\bx}{|K|}\cdot\bn_{\sigma\sigma'}+ \oint_{\partial K}(\varphi_\sigma-\varphi_{\sigma'})\big ( \hbbf_\bn(\bu^h,\bu^{h,-})-\bbf(\bar\bu)\cdot\bn\big )\; d\gamma
\end{equation}
so that the second term can be interpreted as a dissipation. The control volume is depicted in figure  \ref{control:DG}.
\begin{figure}[h]
\begin{center}
\includegraphics[width=0.5\textwidth]{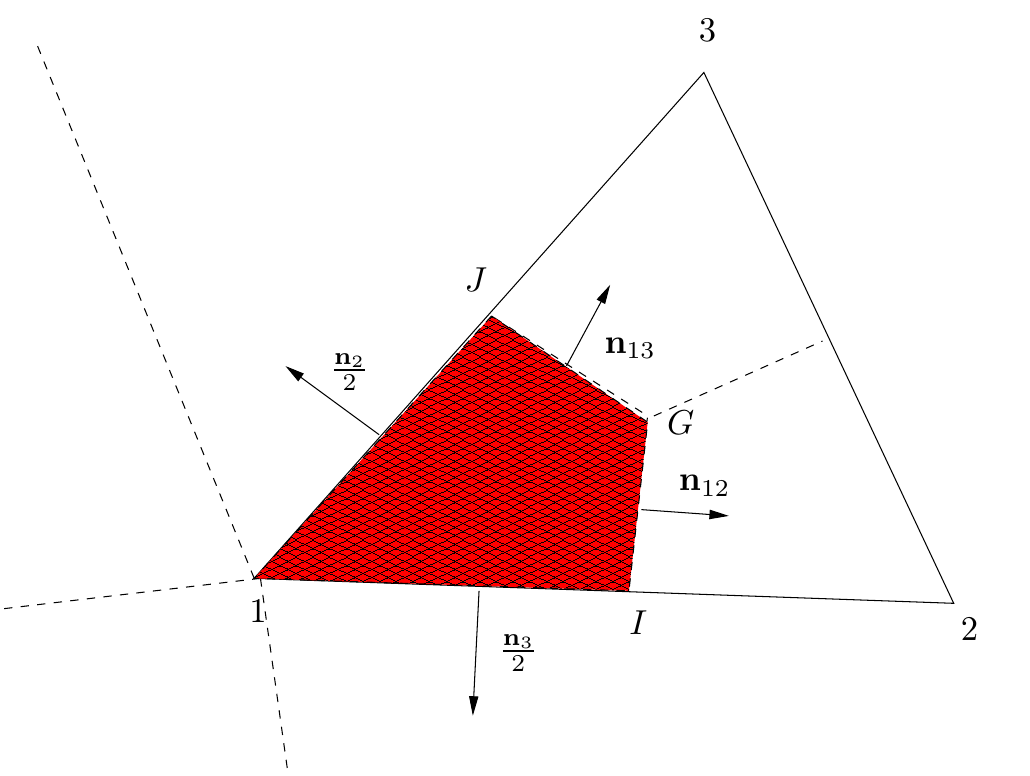}
\end{center}
\caption{\label{control:DG} Representation of the control volume associated to DOF $1$.}
\end{figure}
Referring to figure \ref{control:DG} for the DOF $\# 1$, the flux on the faces $I1J$ is $\oint_{\partial K}\varphi_1 \hbbf_\bn(\bu^h,\bu^{h,-})\; d\gamma$. In order to respect some geometrical assignment, the flux on $1I$ is set to
$$\oint_{1I}\varphi_1 \hbbf_\bn(\bu^h,\bu^{h,-})\; d\gamma$$
and on $1J$,
$$\oint_{1J}\varphi_1 \hbbf_\bn(\bu^h,\bu^{h,-})\; d\gamma.$$

%% file: entropy_new.tex
\section{Entropy dissipation}\label{sec:entropy}
In this section, we consider the system version of \eqref{eq1}. Our results on the flux are similar, since we never have used we were dealing with residual belonging to $\R$ or to some $\R^p$.
\subsection{The 1 D case revisited}
We start by recalling Tadmor's work \cite{TadmorVieux,TadmorActa}.
Let us start from a finite  volume scheme semi-discretized in time:
$$\Delta x \dfrac{d\mathbf{v}_i}{dt}+\hbbf_{i+1/2}-\hbbf_{i-1/2}=0.$$
If $\mathbf{v}$ is the entropy variable, we have:
$$\Delta x \langle \mathbf{v}_i,\dfrac{d\mathbf{u}_i}{dt}\rangle+\langle \mathbf{v}_i,\hbbf_{i+1/2}\rangle-\langle \mathbf{v}_i,\hbbf_{i-1/2}\rangle=0$$
Then
$$\langle \mathbf{v}_i,\hbbf_{i+1/2}\rangle=\langle \frac{\mathbf{v}_i+\mathbf{v}_{i+1}}{2}, \hbbf_{i+1/2}\rangle + \langle \frac{\mathbf{v}_i-\mathbf{v}_{i+1}}{2}, \hbbf_{i+1/2}\rangle$$
Following Tadmor, we introduce the  potential:
$$\theta=\langle \mathbf{v}, \bbf\rangle -\mathbf{g} $$ where $\bbg$ is the entropy flux, so that the entropy flux is defined by:
$$\hat{\bbg}_{i+1/2}:=\langle \frac{\mathbf{v}_i+\mathbf{v}_{i+1}}{2}, \hbbf_{i+1/2}\rangle-\frac{\theta_i+\theta_{i+1}}{2}$$
and we get
$$\langle \mathbf{v}_i,\hbbf_{i+1/2}\rangle= \hat{\bbg}_{i+1/2}+ \langle \frac{\mathbf{v}_i-\mathbf{v}_{i+1}}{2}, \hbbf_{i+1/2}\rangle-\frac{\theta_i+\theta_{i+1}}{2}$$
$$\langle \mathbf{v}_i,\hbbf_{i-1/2}\rangle= \hat{\bbg}_{i-1/2}+  \langle \frac{\mathbf{v}_i-\mathbf{v}_{i-1}}{2}, \hbbf_{i-1/2}\rangle-\frac{\theta_i+\theta_{i-1}}{2}.$$
Thus,
$$\Delta x \langle \mathbf{v}_i,\dfrac{d\mathbf{u}_i}{dt}\rangle+\hat{\bbg}_{i+1/2} -\hat{\bbg}_{i-1/2}=
\bigg ( \langle \frac{\mathbf{v}_{i+1}-\mathbf{v}_{i}}{2}, \hbbf_{i+1/2}\rangle-\frac{\theta_{i+1}-\theta_i}{2}\bigg ) +\bigg (  \langle \frac{\mathbf{v}_{i}-\mathbf{v}_{i-1}}{2}, \hbbf_{i-1/2}\rangle-\frac{\theta_{i}-\theta_{i-1}}{2}\bigg ).$$
This leads to the definition of entropy stable schemes:
\begin{definition}[Tadmor \cite{TadmorVieux,TadmorActa}]
A scheme is entropy dissipative if for any $j$, 
$$\langle \frac{\mathbf{v}_{j+1}-\mathbf{v}_{j}}{2}, \hbbf_{j+1/2}\rangle-\frac{\theta_{j+1}-\theta_j}{2}\leq 0$$
and entropy stable if we have an equality.
\end{definition}

In residue form,  we have
$$\Delta x \dfrac{d\mathbf{u}_i}{dt}+\phi_{i}^{i+1/2}+\phi_{i}^{i-1/2}=0$$
with
$$\phi_{i}^{i+1/2}=\hbbf_{i+1/2}-\bbf_i, \qquad \phi_{i}^{i-1/2}=\bbf_i-\hbbf_{i-1/2}$$ so that for any $j$
$$\phi_{j}^{j+1/2}=\hbbf_{j+1/2}-\bbf_j, \qquad \phi_{j+1}^{j+1/2}=\bbf_{j+1}-\hbbf_{j+1/2}$$
If we compute 
$\langle \mathbf{v}_j, \phi_{j}^{j+1/2}\rangle+\langle \mathbf{v}_{j+1}, \phi_{j+1}^{j+1/2}\rangle$ (note  this term is the one formulated in proposition \ref{th:entropy}),
we get, using $\theta_j+\bbg_j=\langle \mathbf{v}_j, \bbf_j\rangle$
\begin{equation*}
\begin{split}
\langle \mathbf{v}_j, \phi_{j}^{j+1/2}\rangle&+\langle \mathbf{v}_{j+1}, \phi_{j+1}^{j+1/2}\rangle=\langle \mathbf{v}_j,\hbbf_{j+1/2}-\bbf_j\rangle +
 \langle \mathbf{v}_{j+1},\bbf_{j+1}-\hbbf_{j+1/2}\rangle \\
&=\langle \mathbf{v}_j-\mathbf{v}_{j+1}, \hbbf_{j+1/2}\rangle - \langle \mathbf{v}_j,\bbf_j\rangle+\langle \mathbf{v}_{j+1},\bbf_{j+1}\rangle\\
&=\bigg ( \langle \mathbf{v}_j-\mathbf{v}_{j+1}, \hbbf_{j+1/2}\rangle- \theta_j+\theta_{j+1}\bigg ) +\bbg_{j+1}-\bbg_j.
\end{split}
\end{equation*}
So the condition 
$$\langle \mathbf{v}_j, \phi_{j}^{j+1/2}\rangle+\langle \mathbf{v}_{j+1}, \phi_{j+1}^{j+1/2}\rangle\geq \bbg_{j+1}-\bbg_j$$
is equivalent to Tadmor's condition 
$$ \langle \frac{\mathbf{v}_{i+1}-\mathbf{v}_{i}}{2}, \hbbf_{i+1/2}\rangle-\frac{\theta_{i+1}-\theta_i}{2}\leq 0.$$
This suggests natural generalisation to the multidimensional case, i.e. the relation \eqref{entropy:1}.
%%%%%%%%%%%%%%%%%%%%%%%%%%
\subsection{The multidimensional case}\label{sec:multiD}
%%%%%%%%%%%%%%%%%%%%%%%%%%
Let us recall the entropy condition \eqref{entropy:1},
$$\sum\limits_{\sigma\in K} \langle \mathbf{v}_\sigma, \Phi_\sigma\rangle \geq \oint_{\partial K}\hbbg_\bn(\mathbf{u}^h,\mathbf{u}^{h,-}) \; d\gamma.
$$
Written like this, it seems that the residuals and the consistent entropy flux can be chosen independently, which is not exactly the case.

From the previous analysis, we have
$$\Phi_\sigma=\sum\limits_{[\sigma,\sigma']}\hbbf_{\sigma\sigma'}+\hbbf_\sigma^b$$
with the condition \eqref{GC:conservation}.
This suggests to choose
$$\hbbf_\sigma^b=\oint_{\partial K}\varphi_\sigma \hbbf_\bn(\mathbf{u}^h,\mathbf{u}^{h,-})\; d\gamma,$$
because \eqref{entropy:1} becomes:
\begin{equation*}
\begin{split}
 \sum\limits_{\sigma\in K} \langle \mathbf{v}_\sigma, \Phi_\sigma\rangle= \oint_{\partial K} \langle\mathbf{v}^h ,\hbbf_\bn(\mathbf{u}^h,\mathbf{u}^{h,-})\rangle\; d\gamma +\sum\limits_{\sigma \in K}\sum\limits_{\sigma>\sigma'} \langle \mathbf{v}_\sigma-\mathbf{v}_{\sigma'},\hbbf_{\sigma\sigma'}\rangle
\geq \oint_{\partial K}\hbbg_\bn(\mathbf{u}^h,\mathbf{u}^{h,-}) \; d\gamma
\end{split}
\end{equation*}
Here we have set $\mathbf{v}^h=\sum\limits_{\sigma\in K} \mathbf{v}_\sigma \varphi_\sigma$. 

We introduce the potential $\theta^h$ in $K$ by 
\begin{equation}\label{potential}
\theta^h:=\sum\limits_{\sigma\in K} \theta_\sigma \varphi_\sigma \text{ with }
\theta_\sigma=\langle \mathbf{v}_\sigma,\bbf(\mathbf{v}_\sigma)\rangle-\bbg(\mathbf{v}_\sigma).
\end{equation}
Then we  define $\hbbg_{\bn}$ by
\begin{equation}
\label{entropy:flux}
\hbbg_{\bn}(\mathbf{u}^h,\mathbf{u}^{h,-})=\langle \{ \mathbf{v}^h\}, \hbbf_\bn(\mathbf{u}^h,\mathbf{u}^{h,-})\rangle-\{\theta^h\}\cdot \bn.
\end{equation}

The numerical flux is defined only on $\partial K$ and $\{a\}$ is the arithmetic average of the left and right states of $a$ on the boundary of $\partial K$.
The condition \eqref{entropy:1} becomes
\begin{equation}
\label{entropy:1:bis}
\sum\limits_{\sigma>\sigma'} \langle \mathbf{v}_\sigma-\mathbf{v}_{\sigma'},\hbbf_{\sigma\sigma'}\rangle+\oint_{\partial K}\theta^h_K\cdot \bn \; d\gamma -\frac{1}{2}\bigg (\oint_{\partial K} \langle [\mathbf{v}^h],\hbbf_\bn(\mathbf{v}^h,\mathbf{v}^{h,-}\rangle \;d\gamma-\oint_{\partial K} [\theta]\cdot \bn\; d\gamma\bigg )\geq 0.
\end{equation}
Here, the jump definition is consistent with Tadmor's definition in the one dimensional case: for any function $w$,
\begin{equation}
\label{jump}
[w]=w_{|K^-}-w_{|K}.
\end{equation}
From this we see that a sufficient condition for local entropy stability is that:
\begin{subequations}
\label{entropy:stable}
\begin{enumerate}
\item In $K$, we have
\begin{equation}\label{entropy:stable:1:bis}
\sum_{\sigma\in K}\langle \mathbf{v}_\sigma,\Psi_\sigma\rangle +\oint_{\partial K}\theta^h_K\cdot \bn \; d\gamma \geq 0,
\end{equation}
where $\Psi_\sigma=\Phi_\sigma-\hbbf_\sigma^b$,
or equivalently
\begin{equation}
\label{entropy:stable:1}
\sum\limits_{\sigma\in K}\sum\limits_{\sigma>\sigma'} \langle \mathbf{v}_\sigma-\mathbf{v}_{\sigma'},\hbbf_{\sigma\sigma'}\rangle+\oint_{\partial K}\theta^h_K\cdot \bn \; d\gamma \geq 0.
\end{equation}

\item On the boundary of $K$ we ask that the numerical flux $\hbbf$ is entropy stable so that  \begin{equation}
\label{entropy:stable:2}
\oint_{\partial K}\bigg ( \langle [\mathbf{v}^h],\hbbf_\bn(\mathbf{v}^h,\mathbf{v}^{h,-})\rangle -[\theta]\cdot \bn\bigg )\; d\gamma
\leq 0.
\end{equation}
Note that this condition is automatically met for a continuous $\mathbf{u}^h$.
\end{enumerate}
\end{subequations}
The condition \eqref{entropy:stable:2} is automatically met is the flux $\hbbf$ is entropy stable in the sense of Tadmor:
\begin{equation}
\label{entropy:stable:3}
\langle [\mathbf{v}^h],\hbbf_\bn(\mathbf{v}^h,\mathbf{v}^{h,-})\rangle-[\theta^h]\cdot \bn\leq 0.
\end{equation}

%Assuming we have an entropy stable scheme, how can \eqref{entropy:stable:1} be met without loosing accuracy.
Note that these conditions do not make any assumptions on the quadrature formulas on the boundary of $K$ or in $K$. This is in contrast with the conditions on SAT-SBP schemes \cite{Gassner,Zingg,ShuEntropy,Mishra}.
\begin{remark}
Starting from a consistent flux $\hbbf$, a simple way to construct a numerical flux $\hbbf'$ that satisfies \eqref{entropy:stable:3} is to consider:
$$\hbbf'_\bn(\mathbf{v}^h,\mathbf{v}^{h,-})=\hbbf_\bn(\mathbf{v}^h,\mathbf{v}^{h,-})+\alpha\big (\mathbf{v}^h-\mathbf{v}^{h,-}\big )$$
with $\alpha$ chosen so that \eqref{entropy:stable:3} holds true. If the original flux is Lipschitz continuous, this is always possible.

Hence the satisfaction of \eqref{entropy:stable:3} is not an issue. Note this does not spoil the accuracy conditions \eqref{eq: residual accuracy}. Given a numerical flux, it is always possible to construct residuals that satisfies the conservation relation with that given flux. In the appendix, we show how to proceed for discontinuous representations. The next paragraph shows how to enforce a local entropy condition, in general.
\end{remark}
%%%%%%%%%%%%%%%%%
%\begin{remark}
We can rework the relation \eqref{entropy:stable:1} in order to show some links with the recent paper \cite{ShuEntropy}.
Using the flux definitions, we can rewrite 
$$\sum\limits_\sigma\langle\mathbf{v}_\sigma,\Psi_\sigma\rangle+\oint_K\theta_\bn d\gamma$$ as
\begin{equation*}
\begin{split}
\frac{1}{2}\sum\limits_{\sigma>\sigma'} \big ( \langle\mathbf{v}_\sigma-\mathbf{v}_{\sigma'}, \hbbf_{\sigma,\sigma'}\rangle-\big (\theta_\sigma-\theta_{\sigma'}\big )\cdot \bn_{\sigma,\sigma'}\big )+\frac{1}{2}\sum\limits_{\sigma> \sigma'}\big (\theta_\sigma-\theta_{\sigma'}\big )\cdot \bn_{\sigma,\sigma'}
\end{split}
\end{equation*}
Then, we see that
$$\frac{1}{2}\sum\limits_{\sigma>\sigma'}\big (\theta_\sigma-\theta_{\sigma'}\big )\cdot \bn_{\sigma\sigma'}=\sum\limits_{\sigma\in \partial K} \theta_\sigma\cdot \normal_\sigma\; d\gamma=\oint_{\partial K} \theta^h_K\; d\gamma.$$
This relation is the motivation for defining $\theta^h$ in \eqref{potential}. 
Thanks to this, we can write
the condition as:
$$
{\frac{1}{2}\sum\limits_{\sigma >\sigma'} \big ( \langle\mathbf{v}_\sigma-\mathbf{v}_{\sigma'},
 \hbbf_{\sigma,\sigma'}\rangle-\big (\theta_\sigma-\theta_{\sigma'}\big )\cdot \bn_{\sigma\sigma'}\big )}+
{\oint_{\partial K} \theta(\mathbf{v}^h)\cdot\bn\; d\gamma-\sum\limits_{\sigma\in \partial K} \theta_\sigma\cdot \normal_\sigma}\geq 0.$$%+ \bar\alpha \mathcal{D}(\bv_h)\geq 0.$$
with
$$\frac{1}{2}\sum\limits_{\sigma> \sigma'} \big ( \langle\mathbf{v}_\sigma-\mathbf{v}_{\sigma'},
 \hbbf_{\sigma,\sigma'}\rangle-\big (\theta_\sigma-\theta_{\sigma'}\big )\cdot \bn_{\sigma\sigma'}\big )\geq 0.$$%+ \bar\alpha\mathcal{D}(\bv_h)\geq 0.$$
 We see that, as in \cite{ShuEntropy}, if the fluxes  $\hbbf_{\sigma,\sigma'}$ are entropy stable, we get entropy stability at the element level.

%\end{remark}

%% file: DG_RDS.tex
\section{A DG RDS scheme}\label{DG_RDS}
Let us consider problem \eqref{eq1} defined on 
$\Omega\subset \R^2$. In this case, the approximation can be discontinuous across edges: $\bu^h\in \mathcal{V}_h$.

In a first step, we consider a conformal triangulation of  $\Omega$ using triangles. This is not essential but simplifies a bit the text.
 The 3D case can be dealt with in a similar way.

In  $K$, we say that the degrees of freedom are located at the  vertices, and we represent the approximated solution in 
 $K$ by the degree one interpolant polynomial at the vertices of $K$. 
Let us denote by  $\bu^h$ this piecewise linear approximation, that is in principle discontinuous at across edges. 
In the following, we use the notations described in Figure
 \ref{DGRDS:fig1}.
\begin{figure}[h]
\begin{center}
\includegraphics[width=0.5\textwidth]{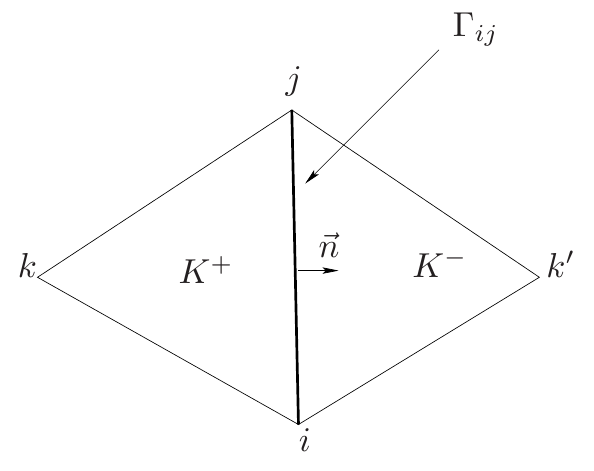}
\end{center}
\caption{\label{DGRDS:fig1} Geometrical elements for defining the scheme.}
\end{figure}

In  \cite{abgrall:shu}, the degrees of freedom are located at the midpoint of the edges that connect the centroid of $K$
 and its vertices. This choice was motivated by  the fact that the 
$\PP^1$ basis functions associated to these nodes are orthogonal in  $L^2(K)$. This property enables us to reinterpret
the DG schemes as RD schemes, and hence to adapt the stabilization techniques of RD to DG. In particular, we are able to
enforce a $L^\infty$ stability property. However, this method was a bit complex, and it is not straightforward to
 generalize it to more general elements than triangles.

The geometrical idea behind the  version  that we describe now  is to forget the RD interpretation of the DG scheme and
 to let the
geometrical localization of the degrees of freedom  move to the vertices of the element.

With this in mind, we define two types of total  residuals:
\begin{itemize}
\item A total residual per element $K$
$$
\Phi^K(\bu^h)=\oint_{\partial K} \bbf(\bu^h)\cdot \bn \; d\gamma.
$$
\item A total residual per edge $\Gamma$, i.e.
$$\Phi_{\Gamma}(\bu^h)=\oint_{\Gamma} \big [ \bbf(\bu)\cdot{\bn} \big ] \; d\gamma
$$
where $[\bbf(\bu)\cdot{\bn}]$ represents the jump of  the function $\bbf(\bu)\cdot \bn$ across  $\Gamma$. Here, if 
$\bn$ is the outward unit normal to $K$ (see figure \ref{DGRDS:fig1}), 
which enables us to define a right side and a left side. Hence we set
$$[\bbf(\bu)\cdot\bn]= ( \bbf(\bu_R)-\bbf(\bu_L) )\cdot \bn.$$
We  notice that  $\Phi_\Gamma$ only depends on the values of  $\bbu$ on each side of $\Gamma$.
\end{itemize}

The idea is to split the total residuals into sub-residuals so that a monotonicity preserving scheme can be defined.
Here, we choose the Rusanov scheme, but other choices could be possible.
Thus we consider
\begin{subequations}\label{DGRDS:LxF}
\begin{itemize}
\item For the element  $K$ and any vertex $\sigma\in K$,
\begin{equation}\label{DGRDS:LxFK}\Phi_{\sigma}^K= \dfrac{\Phi^K}{3} +{\alpha_K} ( \bbu_\sigma -\overline{\bbu})\end{equation} 
with 
$$\overline{\bbu}=\frac{1}{3}\sum\limits_{\sigma'\in K} u_{\sigma'}, $$
and $\alpha_K\geq \max\limits_{\bx\in K} ||\nabla \bbf(\bbu^h(\bx))|| $ where $||~.~|| $ is any norm in $\R^2$, for example the Euclidean norm.
\item and for the edge $\Gamma$, any $\sigma\in \Gamma$,
\begin{equation}\label{DGRDS:LxFGamma}\Phi_\sigma^\Gamma(\bu^h)= \dfrac{\Phi^\Gamma(\bu^h)}{4} +\alpha_\Gamma ( \bbu_{\sigma }-\overline{\bu})\end{equation}
with
$$\overline{\bu}=\frac{1}{4}\sum\limits_{\sigma'\in K^+\cup K^-} \bu_{\sigma'}$$
 where  and $\alpha_\Gamma\geq \max\limits_{K=K^+,K^-}\max\limits_{\bx\in \partial  K\cap \Gamma} ||\nabla \bbf(\bu^h(\bx))||$, see Figure \ref{DGRDS:fig1} for a definition of $K^\pm$.
\end{itemize}
\end{subequations}
We have the following conservation relations
\begin{equation}
\label{DGRDS:conservation}
\begin{split}
\sum\limits_{\sigma \in K} \Phi_{\sigma}^K(\bu^h)&=\Phi^K(\bu^h),\\
\sum\limits_{\sigma \in \Gamma} \Phi_{\sigma}^\Gamma(\bu^h)&=\Phi^\Gamma(\bu^h)
\end{split}
\end{equation}
The choice $\alpha_K\geq \max\limits_{\bx\in K} ||\nabla \bbf(\bu^h(\bx))||$ and $\alpha_\Gamma\geq \max\limits_{K=K^+,K^-}\max\limits_{\bx\in \partial  K\cup \Gamma} ||\nabla_\bu\bbf(\bu^h(\bx))||$ are justified by the following standard argument.
If we set  $Q=K$ or $\Gamma$, we can rewrite the two residuals as
$$\Phi_\sigma^Q(\bu^h)= \sum\limits_{\sigma'\in Q} c_{\sigma\sigma'}^Q (\bu_{\sigma}-\bu_{\sigma'})$$ with $c_{\sigma \sigma'}^Q\geq 0$ under the above mentioned conditions.
Indeed, using  $\bu^h- {\bu_\sigma}=\sum\limits_{\sigma'\in K} (\bu_\sigma-{\bu_\sigma})\varphi_{\sigma'}$, we get (for $Q=K$ for example)
\begin{equation*}
\begin{split}
\Phi_\sigma^K(\bu^h)& = \dfrac{\Phi^K(\bu^h)}{3} +\alpha_K ( \bu_\sigma-\overline{\bu})\\
&\qquad = \dfrac{1}{3} \oint_{\partial K} \big ( \bbf(\bu^h)-\bbf( \bu_\sigma)\big )\cdot{\bn}  \; d\gamma+\alpha_K ( \bu_\sigma -\overline{\bu}) \\
& \qquad =\sum\limits_{\sigma'\in K} \frac{1}{3} \Bigg [\oint_{\partial K}\bigg (\int_0^1 \nabla \bbf(s \bu^h+(1-s) \bu_\sigma)\varphi_{\sigma'}(\bx)\;ds\bigg )  \cdot{\bn} \; d\gamma  -\alpha_K\Bigg ] (\bu_{\sigma}-\bu_{\sigma'})
\end{split}
\end{equation*}
which proves the result.

Using standard arguments, as defining  $u^h$ as the limit of the solution of 
\begin{equation}\label{iterative}
\bu_{\sigma}^{n+1}=\bu_{\sigma}^n-\omega_{\sigma} \Bigg ( \sum\limits_{K, \sigma\in K} \Phi_{\sigma}^K(\bu^{h,n}) + \sum\limits_{\Gamma, \sigma\in \Gamma}\Phi_{\sigma}^\Gamma(\bu^{h,n})\Bigg  )
\end{equation}
with
$$
\omega_{\sigma} \bigg ( \sum\limits_{K, \sigma\in K} c_{\sigma\sigma'}^K  +\sum\limits_{\Gamma, \sigma'\in \Gamma} c_{\sigma\sigma'}^\Gamma \bigg ) \leq 1,
$$
we see that we have a maximum principle.

It is possible to construct a scheme that is formally second order accurate by setting 
\begin{equation}
\label{DGRDS:O2}
\Phi_{\sigma}^{K,\star}(\bu^h) =\beta_\sigma^K \Phi^K(\bu^h) \text{ and } \Phi_{\sigma}^{\Gamma,\star} (\bu^h)=\beta_\sigma^\Gamma \Phi^K(\bu^h)
\end{equation}
with
$$x_\sigma^K=\dfrac{\Phi_{\sigma}^K(\bu^h)}{\Phi^K(\bu^h)}, \qquad x_\sigma^\Gamma =\dfrac{\Phi_{\sigma}^\Gamma(\bu^h)}{\Phi^\Gamma(\bu^h) },$$
and \begin{equation}
\label{DGRDS:O2.2}
\beta_\sigma^K=\dfrac{ \max(x_\sigma^K,0) }{\sum\limits_{\sigma'\in K} \max(x_{\sigma'}^K,0) }, \qquad
\beta_\sigma^\Gamma=\dfrac{ \max(x_\sigma^\Gamma,0)}{\sum\limits_{\sigma'\in K} \max(x_{\sigma'}^\Gamma,0)}.
\end{equation}

As in the ``classical'' RD framework,  the coefficients $\beta$ are well defined thanks to the conservation
 relations \eqref{conservation:K}. The scheme is
written as  \eqref{DGRDS:residu:1} where the residuals  $\Phi_{\sigma}^{K}(\bu^h)$ (resp. $\Phi_{\sigma}^{\Gamma}(\bu^h)$) are replaced
 by  $\Phi_{\sigma}^{K,\star}(\bu^h)$ (resp. $\Phi_{\sigma}^{\Gamma,\star}(\bu^h)$.

The solution $\bu^h$ is defined:  find  $u^h$ linear in each triangle  $K$
 such that for any  degree of freedom $\sigma$ (i.e. vertex of the triangulation),
\begin{equation}
\label{DGRDS:residu:1}
\sum\limits_{K, \sigma\in K} \Phi_{\sigma}^{K,\star} (\bu^h)+ \sum\limits_{\Gamma, \sigma\in \Gamma}\Phi_{\sigma}^{\Gamma,\star}(\bu^h) =0.
\end{equation}
%We specify later the boundary conditions.
We have a first order approximation just by replacing the "starred" residuals by the first order ones. The system \eqref{DGRDS:residu:1} is solved by an iterative method such as\eqref{iterative}.